\documentclass[11pt]{article}

\usepackage[margin=1.2in]{geometry}

\usepackage{amsmath}
\usepackage{amsthm}
\usepackage{graphicx}
\usepackage{color}
\usepackage{amsfonts}
\usepackage{amssymb}
\usepackage{amscd}
\usepackage{url}

\newcommand{\re}{\mathbb{R}}

\newcommand{\Z}{\mathbb{Z}}

\newcommand{\half}{\frac{1}{2}}

\newcommand{\eps}{\epsilon}

\def\af{\alpha}
\def\bt{\beta}

\newcommand{\reff}[1]{(\ref{#1})}

\newcommand{\mc}[1]{\mathcal{#1}}

\newcommand{\bdes}{\begin{description}}
\newcommand{\edes}{\end{description}}

\newcommand{\bal}{\begin{align}}
\newcommand{\eal}{\end{align}}

\newcommand{\bnum}{\begin{enumerate}}
\newcommand{\enum}{\end{enumerate}}

\newcommand{\bit}{\begin{itemize}}
\newcommand{\eit}{\end{itemize}}

\newcommand{\bea}{\begin{eqnarray}}
\newcommand{\eea}{\end{eqnarray}}
\newcommand{\be}{\begin{equation}}
\newcommand{\ee}{\end{equation}}

\newcommand{\baray}{\begin{array}}
\newcommand{\earay}{\end{array}}

\newcommand{\bsry}{\begin{subarray}}
\newcommand{\esry}{\end{subarray}}

\newcommand{\bca}{\begin{cases}}
\newcommand{\eca}{\end{cases}}

\newcommand{\bcen}{\begin{center}}
\newcommand{\ecen}{\end{center}}

\newcommand{\bbm}{\begin{bmatrix}}
\newcommand{\ebm}{\end{bmatrix}}

\newcommand{\bmx}{\begin{matrix}}
\newcommand{\emx}{\end{matrix}}

\newcommand{\bpm}{\begin{pmatrix}}
\newcommand{\epm}{\end{pmatrix}}

\newcommand{\btab}{\begin{tabular}}
\newcommand{\etab}{\end{tabular}}

\theoremstyle{plain}

\newtheorem{theorem}{Theorem}[section]
\newtheorem{thm}[theorem]{Theorem}

\newtheorem{prop}[theorem]{Proposition}
\newtheorem{lem}[theorem]{Lemma}

\newtheorem{cor}[theorem]{Corollary}
\newtheorem{corollary}[theorem]{Corollary}

\theoremstyle{definition}
\newtheorem{example}[theorem]{Example}
\newtheorem{exm}[theorem]{Example}

\newtheorem{remark}[theorem]{Remark}
\newtheorem{quest}[theorem]{Question}

\renewcommand{\subsection}[1]{
    \stepcounter{subsection}
    \settowidth{\hangindent}{\bf\thesubsection.~}
    \hangafter=1
    \bigskip\bigskip\noindent
    {\bf\hbox{\thesubsection.~}#1}\par
    \nobreak
    \medskip
}

\begin{document}

\title{ Positivity of Riesz Functionals and Solutions of Quadratic and
Quartic Moment Problems}
\author{
Lawrence Fialkow\footnote{Department of Computer Science, State University of New York,
New Paltz, New York 12561. Email: fialkowl@newpaltz.edu.
Research was partially supported by NSF Grant DMS-0758378}
\,\, and Jiawang Nie\footnote{Department of Mathematics, University of California San Diego,
9500 Gilman Drive, La Jolla, CA 92093.
The research was partially supported by NSF grants
DMS-0757212, DMS-0844775 and Hellman Foundation Fellowship.
Email: njw@math.ucsd.edu}
}

\maketitle

\begin{abstract}
We employ positivity of Riesz functionals to establish
representing measures (or approximate representing measures) for
truncated multivariate moment sequences.
For a truncated moment sequence $y$,
we show that $y$ lies in the closure of truncated moment sequences
admitting representing measures supported in
a prescribed closed set $K \subseteq \re^n$
if and only if the associated Riesz functional $L_y$ is $K$-positive.
For a determining set $K$, we prove that
if $L_y$ is strictly $K$-positive, then
$y$ admits a representing measure supported in $K$.
As a consequence, we are able to solve the truncated $K$-moment problem of degree $k$
in the cases:
(i) $(n,k)=(2,4)$ and $K=\re^2$;
(ii) $n\geq 1$, $k=2$, and $K$ is
defined by one quadratic equality or inequality.
In particular, these results solve the truncated moment
problem in the remaining
open cases of Hilbert's theorem on sums of squares.
\end{abstract}

\bigskip
\noindent
{\bf Keywords:} truncated moment sequence,
Riesz functional, (strict) $K$-positivity,
determining set, moment matrix, representing measure

\bigskip
\noindent
{\bf AMS subject classifications:}  47A57, 44A60, 47N40, 47A20

\section{Introduction}

Denote by $\Z_{+}$ the set of nonnegative integers
and let $|\af| = \af_1+\cdots+\af_n$ for
$\af \equiv (\af_1,\ldots,\af_n) \in \Z_{+}^n$.
Let $y=(y_\af)_{\af\in \Z_{+}^n, |\af| \leq k}$ be a real multisequence of
 degree $k$ in $n$ variables (also referred to as
a {\it{truncated moment sequence}}),
and let $K\subseteq \re^n$ be a closed set.
The {\it{truncated $K$-moment problem}}
 of degree $k$ concerns conditions on $y$ such that
it has a {\it $K$-representing measure}, i.e., a
positive Borel measure $\mu$ on $\re^n$, supported in $K$, such that
\begin{equation}
y_\af = \int_{\re^n} x^\af d \mu(x), \quad
\forall \, \af\in \Z_{+}^n : |\af| \leq k.
\end{equation}
(Here, $x^\af =
   x_1^{\af_1}\cdots x_n^{\af_n}$ for $x\equiv(x_1,\ldots,x_n) \in \re^n$.)
For $K= \re^n$, we refer to (1.1) simply as the {\it{truncated moment
problem}} and to  $\mu$ as a {\it{representing measure}}. Let
$\mathcal{P}_{k}\subset \re [x_{1},\ldots,x_{n}]$ denote the
polynomials of degree at most $k$. Corresponding to the sequence $y$
of degree $k$
is the {\it{Riesz functional}} $L_{y}:\mathcal{P}_{k} \longrightarrow \re$
defined by
$$
L_y(p) = \sum_{\af\in \Z_{+}^n: |\af| \leq k}  p_\af y_\af, ~~~~~\, \forall \,
p\equiv \sum_{\af\in \Z_{+}^n: |\af| \leq k}  p_\af x^\af \in \mathcal{P}_{k}.
$$
$L_y$ is said to be {\it K-positive} if
$$
L_y(p) \geq 0, \quad \forall \,
p\in \mc{P}_{k},~ p|_{K} \geq 0.$$
Further, $L_{y}$ is {\it{strictly $K$-positive}}
if  $L_{y}$ is $K$-positive and
$$L_y(p) > 0,  \quad \forall \,
p\in \mc{P}_{k},~ p|_{K} \geq 0,~ p|K \not\equiv 0.$$
 For
$K=\re^{n}$ we say simply that $L_{y}$ is {\it{positive}} or {\it{strictly
positive}}.
$K$-positivity is a necessary condition for $K$-representing measures,
for if $\mu$ is a $K$-representing measure
and $p\in  \mc{P}_{k}$ with $p|_K \geq 0$,
then $L_{y}(p) = \int_{K} p~d\mu \ge 0$.  The proof of
Tchakaloff's Theorem \cite{Tch} shows that if $K$ is compact, then
$K$-positivity is actually sufficient for
$K$-representing measures,
but this is not so
in general
(see below).
Nevertheless, in \cite{CF09} R.E. Curto
and the first-named author obtained the following
 solution to the truncated $K$-moment problem expressed
 in terms of $K$-positivity.
\begin{theorem} [Theorem 1.2, \cite{CF09}] A multisequence
 $y$ of degree $2d$ or $2d+1$ admits
a $K$-representing measure if and only if $y$ can be extended to
a sequence $\tilde{y}$ of degree $2d+2$ such that $L_{\tilde{y}}$ is
$K$-positive.
\end{theorem}

A significant issue associated with Theorem 1.1 is that in general
it is quite difficult to establish that $L_{y}$ or $L_{\tilde{y}}$
is $K$-positive.
We show in Section 2 (Theorem 2.2)
that $L_{y}$ is $K$-positive if and only if
$\lim_{m\to\infty} \|y - y^{(m)} \| = 0$
for a sequence $\{y^{(m)}\}$ in which each
truncated moment sequence $y^{(m)}$ has a $K$-representing measure $\mu^{(m)}$.
In this case, for each $\alpha$, we have
$y_{\alpha} = \lim_{m\to\infty}\int_{K} x^{\alpha} d\mu^{(m)}(x)$, and we say that
$\{\mu^{(m)}\}$ is a sequence of
{\it{approximate representing measures}} for $y$.
This leads us to identify some cases of interest,
including certain multivariate quadratic and quartic moment problems,
in which we can utilize such approximating sequences
to establish $K$-representing measures for $y$
or $K$-positivity for $L_{y}$.
To explain our results further,
consider $K= \re^{n}$.
For $k=2d$, the moment sequence $y$ is associated with
the $d$-th order {\it{moment matrix}} $M_d(y)$ defined by
\[
M_d(y) = (y_{\af+\bt})_{(\af,\bt)\in \Z_{+}^n \times \Z_{+}^n: |\af|, |\bt| \leq d}.
\]
(We sometimes refer to a representing measure for $y$ as a
representing measure for $M_{d}(y)$.)
A basic necessary condition for
positivity of $L_{y}$ (and hence for the existence of a representing measure)
 is that   $M_{y}$ be positive semidefinite
($ M_d(y) \succeq 0$).
To see this, observe that $M_d(y)$ is uniquely determined by
the relation
\begin{equation}  \label{eq:Mpq}
    \langle M_{d}(y)\hat{p},\hat{q}\rangle =  L_{y}(pq) \quad p,q\in \mc{P}_{d},
\end{equation}
where $\hat{r}$ denotes the coefficient vector of $r\in \mathcal{P}_{d}$
relative to the basis for
  $\mathcal{P}_{d}$ consisting
of the monomials in degree-lexicographic order.
Thus, if $L_{y}$ is positive, then
           $\langle M_{d}(y)\hat{p},\hat{p}\rangle =  L_{y}(p^{2})
 \ge 0$. It is known that if $L_{y}$ is positive and
 $M_{d}(y)$ is singular, then $y$ need not have a representing measure;
 the simplest such example occurs with $n=1$, $d=2$ and $M_{2}(y)$  of
 the form
\[
M_2(y) =
\bbm
a & a & a  \\
a & a & a  \\
a & a & b
\ebm,
\]
with $b>a>0$ (cf. \cite[Example 2.1]{CF09}).
Nevertheless, the following question, essentially asked in
(\cite[Question 2.9]{CF09}),
remains unsolved.
\begin{quest}
Let $k=2d$.
If  $L_{y}$ is $K$-positive and $M_{d}(y)$ is positive definite,
 does $y$
have a $K$-representing measure; equivalently, does $L_{y}$
admit a
$K$-positive extension $L_{\tilde{y}}:\mathcal{P}_{2d+2} \longrightarrow
\re$?
\end{quest}
In the sequel we say that $K$ is a {\it{determining set (of degree $k$)}} if
whenever $p\in \mathcal{P}_{k}$ and $p|K \equiv 0$, then $p \equiv 0$ (i.e.,
$p(x) = 0 \quad \forall x\in \re^n$);
sets $K$ with
nonempty interior are clearly determining. It follows readily from \reff{eq:Mpq} that
if $K$ is a determining set and $L_{y}$ is strictly $K$-positive, then
$M_{d}(y)\succ 0$.
Our main
tool in establishing $K$-representing measures is the following result,
which complements Theorem 1.1 and partially answers Question 1.2.
\begin{theorem} Suppose $K$ is a determining set of degree $k$ and
let $y$ be a truncated moment sequence of degree $k$ in $n$ variables.
If $L_{y}$ is strictly $K$-positive,
then $y$ admits a $K$-representing measure.
\end{theorem}
\noindent To discuss concrete applications of Theorem 1.3, we consider
 the following property:
\begin{quote}
$(H_{n,d})$ Each $p\in \mathcal{P}_{2d}$
admits a sum-of-squares decomposition,
$p = \sum~ p_{i}^{2}$,
\newline $~~~~~~~~~$
for certain polynomials $p_{i}\in \mathcal{P}_{d}$ (which depend on $p$).
\end{quote}
If $(H_{n,d})$ holds and we set $k=2d$,
then positivity for $L_{y}$ is equivalent to
positivity of $M_{d}(y)$; indeed, in this case,
if $M_d(y)\succeq 0$ and $p\in \mathcal{P}_{2d}$ is nonnegative on $\re^{n}$, then $L_{y}(p)
= \sum L_{y}(p_{i}^{2})
= \sum  \langle M_{d}(y)\hat{p_{i}},\hat{p_{i}}\rangle   \ge 0$.
A well-known theorem of Hilbert (cf. \cite{Rez00,Rez07})
 shows that $(H_{n,d})$ holds
if and only if $n=1$, $n=d=2$, or $n>1$ and $d=1$.
In these cases, whether or not $y$ has a representing measure,
Theorem~\ref{thm:pos=>clos-rep} (cf. Section 2)
implies that if $M_{d}(y)\succeq 0$, then $y$ has
a sequence of approximate representing measures.
For $n=1$,
the truncated moment problem has been solved (cf. \cite{CF96}):
a multisequence $y$ of degree  $2d$ has a representing measure
if and only if $M_{d}(y)$ is positive semidefinite and {\it{recursively
generated}}
 (see below for terminology concerning moment matrices). In the
sequel we address the truncated moment problem in the other cases
covered by Hilbert's theorem.

Consider first the bivariate quartic moment problem $(n=d=2)$.
For the case when $M_{2}(y)$ is singular,
concrete necessary and sufficient conditions
for representing measures are known (cf. \cite{CF02,CF052}):
$y$ has a representing measure if and only if
\begin{equation}
M_{2}(y)\succeq 0,~M_{2}(y)~is
~recursively~ generated,~ and ~
rank~  M_{2}(y) \le card~ \mathcal{V}(M_{2}(y)),
\end{equation}
where $\mathcal{V}(M_{2}(y))$
is the {\it{algebraic variety}}
associated to  $M_{2}(y)$ (see definition~\reff{def:momvar})
and $card$ denotes the cardinality of a set.
When $2$ is replaced by $d$,
the conditions of (1.3) apply
more generally to any bivariate sequence $y$ of degree $2d$
for which $M_2(y)$ is singular, i.e., the first 6 columns
of $M_2(y)$ are dependent
(cf. \cite[Theorem 1.2]{CF052}).
Subsequent to
\cite{CF02}, the case  $M_{2}(y)\succ 0$ has been open (cf. \cite{LL}).
In this case, it is easy to find a moment matrix
 extension $M_{3}(\tilde{y})\succ 0$,
but an example of \cite{CF98}
shows that for such $\tilde{y}$, $L_{\tilde{y}}$
need not be positive, so Theorem 1.1 cannot be applied to yield
a representing measure for $y$.
Instead, in Section 3 we will use Theorem 1.3, together
 with Hilbert's theorem, to establish that
such $y$ does indeed have
a representing measure.
This provides a positive answer to Question 1.2 for
$n = d = 2$, with $K = \re^2$.

Consider next the case of the multivariate quadratic moment problem,
where $n\ge 1$ and $d=1$. For $n=1,~2$, it was shown in \cite{CF96} that
if $M_{1}(y)\succeq 0$, then y has a
$rank~   M_{1}(y)$-atomic
representing measure,
and in Section 4,  Theorem~\ref{thm:deg2mom},
we prove the same result for $n\ge 1$.
In the sequel, let $\mathcal{R}_{n,k}(K) $
denote the convex set of $n$-variable moment sequences of degree $k$
which admit $K$-representing measures, and let
$\overline{\mathcal{R}_{n,k}(K)}$ denote the closure
of  $\mathcal{R}_{n,k}(K) $
in $\re^{\eta}$,
where $\eta = dim~\mc{P}_{k} $.
 Now let $q$ be a quadratic
polynomial, and
define the quadratic variety $E(q)=\{x\in\re^n:q(x)=0\}$
and the quadratic semialgebraic set $S(q):=\{x\in\re^n:q(x)\geq0\}$.
We are interested in determining whether $y$ has a representing measure
supported in $E(q)$ or in $S(q)$.
It is obvious that if $y$ has a representing measure
supported in $E(q)$ (resp., $S(q)$), then
\be \label{2mom:necc}
M_1(y) \succeq 0, \quad L_y(q) = 0 \,\,\, (resp., L_y(q) \geq 0).
\ee
For the case when $S(q)$ is compact,
we will show in Theorem 4.7 that
if $y$ satisfies \reff{2mom:necc},
then $y
\in \mc{R}_{n,2}(E(q))$
(resp., $y \in \mc{R}_{n,2}(S(q))$).
For the general case, we show in
Theorem 4.8 that if (1.4) holds, then $y
\in \overline{\mc{R}_{n,2}(E(q))}$
(resp., $y \in \overline{\mc{R}_{n,2}(S(q))}$).
In Theorem 4.10,
we further show that
if $M_1(y)\succ 0$
and $L_y(q)=0$ (resp., $L_y(q)>0$),
then $y \in \mc{R}_{n,2}(E(q))$ (resp., $y \in\mc{R}_{n,2}(S(q))$);
this result implies an affirmative answer to Question 1.2
for $d=1$ and $K=E(q)$ (resp., $K=S(q)$).

The preceding concrete results all concern the positive cases
of Hilbert's theorem. In some cases where sums-of-squares are not available,
it is still possible to
use a sequence of approximate representing measures to
establish positivity of a functional $L_{y}:\mathcal{P}_{2d}
\longrightarrow \re$. In Example 2.5, for $n=2$, $d=3$, $k=6$, we will
use
this approach to
illustrate a multisequence $y$ of degree 6 such that $L_{y}$ is
positive (whence $M_{3}(y)\succeq 0$),
but $y$ has no representing measure. We believe this is
the first such example in a case where the positivity of $L_{y}$
cannot be established by sums-of-squares, via positivity of
$M_{d}(y)$.

We recall some additional terminology and results
from \cite{CF96,CF05}
concerning moment matrices and representing measures.
Let  $[x]_k$ denote the column vector of all
$n$-variable
monomials up to degree $k$
 in degree-lexicographic order, that is,
\[
[x]_k^T = \bbm 1 & x_1 & \ldots & x_n & x_1^2 & x_1x_2 & \ldots & x_n^k \ebm.
\]
Throughout this paper, the superscript $T$ denotes
the transpose of a matrix or vector.
Note that if $k=2d$ and $\mu$ is a representing measure for $y$, then
\[
M_d(y) =  \int_{\re^2} [x]_d[x]_d^T d \mu(x),
\]
which shows again that $M_{d}(y)\succeq 0$ is a necessary condition
for representing measures.
Moreover, in this case, $card~supp~ \mu \ge rank~M_{d}(y)$ [4]
(where $supp~\mu$ denotes the closed support of $\mu$).
We denote the successive columns of $M_{d}(y)$ by
$$1,~ X_{1},\ldots,~X_{n}, ~ X_1^2,~ X_1X_2,\ldots,~X_{n}^{2},
~\ldots,~X_n^d,\ldots, ~X_n^d.$$ For
$p=\sum_{\af\in \Z_{+}^n: |\af| \leq d}  p_\af x^\af \in \mathcal{P}_{d}$,
we define an element $p(X)$ of the column space of $M_{d}(y)$ by
$$p(X) =    \sum_{\af\in \Z_{+}^n: |\af| \leq d}  p_\af X^\af.$$
 $M_{d}(y)$  is {\it{recursively generated}} if, whenever
 $p\in  \mathcal{P}_{d}$ and $p(X)=0$, then $(pq)(X)=0$
for $q\in \mathcal{P}_{d}$ with $deg~pq\le d$; recursiveness is
a necessary condition for representing measures \cite{CF96}.
The {\it{algebraic
variety}} associated to $M_{d}(y)$ is defined by
\be \label{def:momvar}
\mathcal{V}(M_{d}(y)) := \bigcap_{p\in \mathcal{P}_{d},~p(X)=0}
\{x\in \re^n: p(x) = 0\};
\ee
if $y$ has a representing measure
$\mu$, then $supp~\mu \subseteq   \mathcal{V}(M_{d}(y))$ \cite{CF96}, whence
\be \label{rank<=card}
rank~M_{d}(y)\le card ~    \mathcal{V}(M_{d}(y)).
\ee

Recall that a measure $\nu$ is {\it{$p$-atomic}} if it is of the
form $\nu = \sum_{i=1}^{p} \lambda_{i}\delta_{u_{i}}$, where
$\lambda_{i}>0$ and $\delta_{u_{i}}$ is the unit-mass
measure supported at $u_{i}\in \re^{n}$. For $k=2d$,
a fundamental result of \cite{CF96,CF05} shows that $y$ admits a
$rank~M_{d}(y)$-atomic
representing measure if and only $M_{d}(y)$ is positive semi-definite
and $M_{d}(y)$ admits a {\it{flat}} (i.e., rank-preserving) moment
matrix extension  $M_{d+1}(\tilde{y})$;
in this case $\tilde{y}$ has a unique (and computable)
representing measure,
which is $rank~M_{d}(y)$-atomic, with support precisely
$\mathcal{V}( M_{d+1}(\tilde{y}) )$.
 More generally, $y$ admits
a finitely atomic representing measure if and only if
 $M_{d}(y)$ admits a positive extension $M_{d+m}(\tilde{y})$
 (for some $m\ge 0$), which in turn admits a flat extension
$M_{d+m+1}$ \cite{CF05}. A remarkable result of Bayer and Teichmann \cite{BT}
implies that
a multisequence
$y$ of degree $k$
admits a $K$-representing measure if and only if
$y$ admits a finitely atomic $K$-representing measure $\mu$ (with
$card~supp~\mu \le dim~\mathcal{P}_{k} $), so the preceding
moment matrix criterion provides a complete characterization of the
existence of representing measures when $k=2d$.
This characterization  is
more concrete than the criterion of Theorem 1.1, because it
provides algebraic coordinates for
constructing representing measures, although
precise conditions for flat extensions are presently known
only in special cases.
For the case when $K$ is a closed semialgebraic set,
analogues of the preceding results appear in \cite{CF05}.
The papers \cite{CF02,CF052,F08}
describe various concrete existence theorems for representing
measures based on flat extensions. These results usually assume that
$M_{d}(y)$ is positive semidefinite and {\it{singular}},
so that any representing measure is necessarily supported in
the nontrivial algebraic variety $\mathcal{V}(M_{d}(y))$. By contrast,
for the case when   $M_{d}(y)$ is positive definite,
very few results  are known concerning the existence of representing measures.
Our solutions to the positive definite cases of the bivariate
quartic moment problem
and the multivariate quadratic moment problem
provide two such results.
A notable feature of the proofs of these results
is that they do not rely on flat extension techniques.
For this reason, the results which depend on
Theorem 1.3 (or Lemma~\ref{lem:clos-int}) are purely existential and do
not provide a procedure for explicitly computing
representing measures
(cf. Question~\ref{quest3.5} below).

\bigskip
This paper is organized as follows.
Section~2 contains
an analysis of positivity of Riesz functionals, leading to
a proof of Theorem 1.3.
Section~3 shows that every bivariate quartic moment sequence $y$
with $M_2(y)\succ 0$ admits a representing measure supported in $\re^2$.
Section~4 gives a complete solution of quadratic $K$-moment problems
when
$K = \re^{n}$, or when
$K\equiv S(q)$ or $K\equiv E(q)$
is defined by a quadratic multivariate polynomial $q(x)$.

\section{Positivity, approximation, and representing measures.}
\setcounter{equation}{0}
In this section we will prove Theorem 1.3.
Let
\[
\mc{M}_{n,k} = \left\{
y \equiv (y_\af)_{ \af \in \Z_{+}^n: |\af| \leq k}  \right\},
\]
the set of $n$-variable multisequences of degree $k$, and let
\[
\mc{R}_{n,k}(K) = \left\{
y \in \mc{M}_{n,k} :
y_\af = \int_K x^\af d \mu(x), ~\mu\ge 0,~
\, supp(\mu) \subseteq K
\right\},
\]
the multisequences with $K$-representing measures.
When $K=\re^n$, we  simply write $\mc{R}_{n,k}(\re^n)=\mc{R}_{n,k}$.
Note that $\mc{R}_{n,k}(K)$ is a convex cone in
 $ \mc{M}_{n,k}(K)$  and that
 $\mc{M}_{n,k}$   can be identified with the affine space
$\re^{\eta}$,
where $\eta \equiv dim~\mathcal{P}_{k} =
 \binom{n+k}{k}$.
$\re^\eta $ is equipped with the usual Euclidean norm
$\| \cdot \|$, although we sometimes employ $\|\cdot\|_{1}$ as well.
Note also that for $x\in K$,
the truncated moment sequence $y\equiv [x]_{k}$ is an element of
$\mc{R}_{n,k}(K)$, since $\delta_{x}$ is a $K$-representing measure.
The truncated moment sequence $y$ is said to be in the {\it{interior}} of
$\mc{R}_{n,k}$
if there exists $\eps>0$ such that for any truncated moment sequence
$y^{*}$ having the same degree as $y$,
$y^{*} \in \mc{R}_{n,k}$
whenever $\|y^{*}   - y \| < \eps$.
 Equivalently, the interior of
$\mc{R}_{n,k}$
is defined in the standard way for a subset of the space
$\re^{\eta}$.

Let us begin with a well-known
fact about the interior and closure of convex sets.
%
%

\begin{lem}  \label{lem:clos-int}
If $\mc{C} \subset \re^N$ is a convex set,
then $int(\mc{C}) = int(\overline{\mc{C}})$.
\end{lem}

\noindent
The above lemma is a consequence of Theorem~25.20 (iii) of Berberian \cite{Ber},
which actually applies to convex sets in general topological vector spaces.


In the sequel, let  $\mc{F}_{n,k}(K)$ denote the moment sequences
   $y \in \mc{M}_{n,k}$   having finitely atomic $K$-representing measures.
$\mc{F}_{n,k}(K)$ is clearly a convex subset of  $\mc{R}_{n,k}(K)$,
and the Bayer-Teichmann theorem
\cite[Theorem 2]{BT} \cite[Theorem 5.8]{Lau}
shows that $\mc{F}_{n,k}(K)= \mc{R}_{n,k}(K)$.
The following result, which is implicit in the proof of
\cite[Theorem 2.4]{CF09},
is the basis for our approximation approach to $K$-positivity
for Riesz functionals.

\begin{thm} \label{thm:pos=>clos-rep}
For $y \in \mc{M}_{n,k}$,  the following are equivalent:
\newline \indent i) $L_y$ is $K$-positive;
\newline \indent ii) $y \in \overline{\mc{F}_{n,k}(K)}$.
\newline \indent iii) $y \in \overline{\mc{R}_{n,k}(K)}$.
\end{thm}
\begin{proof}
We begin with $iii) \Longrightarrow i)$.
If $y\in  \mc{R}_{n,k}(K)$, with $K$-representing measure $\mu$,
then $L_{y}$ is $K$-positive; indeed, if $p\in \mathcal{P}_{k}$ and
$p|K\ge 0$, then $L_{y}(p) = \int_{K} pd\mu \ge 0$.
Since the $K$-positive linear functionals form a closed
positive cone in the dual space $\mathcal{P}_{k}^{*}$
(equipped with the usual norm topology),
it follows that if $y\in  \overline{\mc{R}_{n,k}(K)}$, then $L_{y}$ is
$K$-positive.

Since $ii) \Longrightarrow iii)$ is clear, it suffices to show
$i) \Longrightarrow ii)$, which we  prove by  contradiction.
Suppose
$L_{y}$ is $K$-positive, but
 $y \not\in \overline{\mc{F}_{n,k}(K)}$.
Since $\overline{\mc{F}_{n,k}(K)}$ is a closed convex cone in $\re^{\eta}$,
it follows from the
Minkowski separation theorem  \cite[(34.2)]{Ber}
that
there exists a nonzero vector $p\in \re^{\eta}$ such that
\[
p^T y < 0, \quad \text{and} \quad
p^T w \ge 0, \, \forall \, w\in \overline{\mc{F}_{n,k}(K)}.
\]
Now define the nonzero polynomial $\tilde{p}$
in $\mathcal{P}_{k}$ by
\[
\tilde{p}(x) = p^T [x]_k.
\]
Since,
for each $x\in K$,
 the monomial vector $[x]_k$ is an element of
$ \mc{F}_{n,k}$
 (with $K$-representing measure $\delta_{x}$),
 then
$\tilde{p}(x)$ is nonnegative on $K$.
However, we have
\[
L_y(\tilde{p}) = p^T y < 0,
\]
which contradicts the $K$-positivity of $L_y$.
Therefore, we must have
$y \in \overline{\mc{F}_{n,k}(K)}$.
\end{proof}

\begin{lem}  \label{lem:pos-int}
Let $K$ be a determining set of degree $k$ and let
 $y \in \mc{M}_{n,k}$.
If the Riesz functional $L_y$ is strictly $K$-positive,
then there exists $\eps>0$ such that
$L_{\tilde{y}}$ is also strictly $K$-positive
whenever $\|\tilde{y} - y\|_1 < \eps$.
\end{lem}
\begin{proof}
We equip $\mathcal{P}_{k}$ with the norm
$$\| p \| = \max_{\af} |p_{\af}|~~~~
  ( p\equiv \sum_{\af\in \Z_{+}^n: |\af| \leq k}  p_\af x^\af \in \mc{P}_{k}).$$
A sequence $\{p^{(i)}\}$ in $\mathcal{P}_{k}$ that is norm-convergent to
$p\in \mathcal{P}_{k} $          is also
pointwise convergent,  so if $p^{(i)}|K\ge 0$ for each $i$, then $p|K\ge 0$.
It follows that the set
\[
\mc{T} := \left\{
p \in \mc{P}_{k}:
p|_K \geq 0, \| p \| =1
\right\}
\]
is compact.
Note that since $K$ is a determining set,
if $p\in \mc{T}$, then $p|K \not \equiv 0$.
Thus,
 $L_{\tilde{y}}$ is strictly $K$-positive if and only if
$L_{\tilde{y}} (p) >0$ for every $ p\in \mc{T}$.
Since
$\mc{T}$ is compact and
$L_y:\mathcal{P}_{k} \longrightarrow \re$
is a norm-continuous functional on $\mc{T}$,
there exists $\eps >0$ such that
\[
 L_y(p) \geq  2\eps, \quad \forall \, p\in \mc{T}.
\]
For any $p \in \mc{T}$, we have
\[
| L_y(p) - L_{\tilde{y}}(p)| \leq \|y - \tilde{y} \|_1.
\]
So, if $\|y - \tilde{y} \|_1 < \eps$, then
\[
 L_{\tilde{y}} (p) \geq  L_{y} (p) - \|y - \tilde{y} \|_1 \geq \eps >0,
 \quad \forall \, p\in \mc{T},
\]
whence $L_{\tilde{y}}$ is strictly positive.
Thus, the lemma is proved.
\end{proof}

We now prove Theorem 1.3, which  we can restate as follows for convenience.
\begin{thm}  \label{thm:strpos=>rep}
Let $K$ be a determining set of degree $k$.
If $y \in \mc{M}_{n,k}$ and $L_y$ is strictly $K$-positive,
then $y \in \mc{R}_{n,k}(K)$.
\end{thm}
\begin{proof}
By Theorem~\ref{thm:pos=>clos-rep},
we have $y \in \overline{\mc{R}_{n,k}(K)}$.
Lemma~\ref{lem:pos-int} implies that $y$
lies in the interior of $\overline{\mc{R}_{n,k}(K)}$.
Lemma~\ref{lem:clos-int} tells us that
$\overline{\mc{R}_{n,k}(K)}$ and $\mc{R}_{n,k}(K)$
have the same interior.
Therefore we must have
$y \in int(\mc{R}_{n,k}(K)) \subset \mc{R}_{n,k}(K)$.
\end{proof}

Although we believe that the hypothesis that $K$ is a determining
set cannot be omitted from Theorem~\ref{thm:strpos=>rep},
at present we do not have an example
illustrating this.
We next present an example which shows
how a sequence of approximate representing measures
  can be used to establish positivity
of a functional $L_{y}:\mathcal{P}_{2d} \longrightarrow \re$
in a case where $y$ has no representing measure
and the positivity of $L_{y}$ cannot be derived from the
positivity of $M_{d}(y)$ via sums-of-squares arguments.
Let $n=2$ and consider the bivariate moment matrix
$M_{d}(y)$.
Denote the rows and columns by
\[
1,~X_{1},~X_{2},~X_{1}^{2},~X_{1}X_{2},~X_{2}^{2},
\ldots, ~X_{1}^{d},~X_{1}^{d-1}X_{2},~\ldots,X_{1}X_{2}^{d-1},~X_{2}^{d};
\]
then $y_{ij}$ is precisely the entry in row $X_{1}^{i}$, column
$X_{2}^{j}$, the moment corresponding to the monomial $x_{1}^{i}x_{2}^{j}$.

\begin{example} Let $n=2$ and $d=3$. We consider
the general form of a moment matrix $M_{3}(y)$ with a
column relation $X_{2}=X_{1}^{3}$ (normalized with $y_{00} = 1$):
\begin{equation} \label{def:M2.5}
M\equiv M_3(y) =
\bbm
1 & a & b & c & e & d   & b & f & g & x      \\
a & c & e & b & f & g   & e & d & h & j      \\
b & e & d & f & g & x   & d & h & j & k      \\
c & b & f & e & d & h  & f & g & x & u      \\
e & f & g & d & h & j & g & x & u & v      \\
d & g & x & h & j & k & x & u & v & w      \\
b & e & d & f & g & x   & d & h & j & k      \\
f & d & h & g & x & u  & h & j & k & r      \\
g & h & j & x & u & v   & j & k & r & s      \\
x & j & k & u & v & w  & k & r & s & t
\ebm .
\end{equation}
For suitable values of the moment data,
$M$ satisfies the following properties:
\begin{equation} \label{M-rank=9}
 M\succeq 0, \quad X_{2}= X_{1}^{3}, \quad rank~M = 9;
 \end{equation}
 this is
the case, for example, with
\be \label{abfg...rstx}
\baray{c}
a=b=f=g=u=v=w=0,~ c=1,~ e=2,~ d=5,
~h=14, \\
j=42, ~k=132,
~r=429,~ s=1442 ,~ t=4798 ,~ x=0.
\earay
\ee

In \cite{F08} we solved the truncated $K$-moment problem
for $K = \{(x_{1},x_{2})\in \mathbb{R}^{2}: x_{2} = x_{1}^{3}\}$.
In particular,
\cite{F08} provides a numerical test, that we next describe,
 for the existence
of $K$-representing measures whenever $M$ as in \reff{def:M2.5}
satisfies
\reff{M-rank=9}.
From \cite{BT} we know that if $M$ admits a representing measure,
then $M$ admits a finitely atomic measure, and thus $M$ admits a
positive, recursively generated extension $M_{4}(\tilde{y})$.
In any such extension, the moments must be consistent with the
relation $x_{2} = x_{1}^{3}$, so in particular, we must have
$y_{44} = y_{15} (\equiv s)$.
To insure positivity of $M_{4}(\tilde{y})$,
we require a lower bound for the diagonal element
$y_{44}$, which we may derive as in \cite{F08}.
 Let $J$ denote the compression of $M$
obtained  by deleting row $X_{1}^{3}$ and column $X_{1}^{3}$;
thus, $J\succ 0$.
Let us write
$$ J = \bbm N & U \\
            U^{T} & \Delta
\ebm, $$
where $N$ is the compression of $J$ to its first 8 rows and columns,
$U$ is a column vector, and $\Delta \equiv y_{06} (\equiv t) > 0 $.
Consider the corresponding
block decomposition
of $J^{-1}$, which is of the form
$$ J^{-1} = \bbm P & V \\
            V^{T} & \epsilon
\ebm, $$
where $P\succ 0$ and $\epsilon > 0$. In
extension $M_{4}(\tilde{y})$, we have $X_{1}^{4} = X_{1}X_{2}$ and
$X_{1}^{3}Y_{2} = Y_{2}^{2}$, so by moment matrix structure,
after deleting the element in row $X_{1}^{3}$,
the first 8 remaining
elements of column $X_{1}^{2}X_{2}^{2}$
 must be $W\equiv
(h,~ x, ~u, ~j, ~k, ~r,  ~v, ~w)^{T}$.
Let $\omega = \langle PW,W \rangle$ and define
\begin{equation}
    \psi(y) := \frac{\omega \epsilon - \langle V,W \rangle^{2}}{\epsilon}.
\end{equation}
In \cite{F08} we showed that in  $M_{4}(\tilde{y})$ we must have
$y_{44}\ge \psi(y)$, and \cite[Theorem 2.4]{F08} implies that
$M$ has a representing measure if and only
$y_{15} \equiv s > \psi(y)$.

A calculation shows that for $M$ as in \reff{def:M2.5}
and satisfiying \reff{M-rank=9},
with appropriate values of the moment data we can
 also have
$\psi(y)$  {\it {independent}} of $s$ and $t$. This is the case, for example,
if we modify \reff{abfg...rstx} so that $x=\frac{1}{10}$, $r=600$,
 $s$ is arbitrary and $t$ is chosen sufficiently large so as
to preserve positivity and the property $rank~M_{3}(y) =9$.
More generally, this is the case
if we modify \reff{abfg...rstx} so that
 $x,k,u,v,w,r,s,t$
are chosen, successively, to maintain positivity and the
$rank~M = 9$ property.
(We conjecture that whenever $M_{3}(y)$ satisfies \reff{M-rank=9},
 then $\psi(y)$ is independent of
$s$ and $t$.)
 For any such $M$, with $\psi(y)$ independent of $s$  and $t$,
we now specify $s\equiv y_{1,5} = \psi(y)$ and we
adjust $t$ (if necessary)
so that $M$ continues to be positive with $rank~M = 9$.
(For a specific example, we may modify \reff{abfg...rstx}
so that $x= \frac{1}{10}$, $r=600$,
$s \equiv \psi(y) = \frac{526337068574699}{741609900} \approx 709722$, and
$t\ge
11319100143$ (cf. \cite[Example 3.2]{F08}.)

We claim that $L_{y}$ is $K$-positive for
$K = \{(x_1, x_2)\in \mathbb{R}^{2}: x_2 = x_1^3\}$, and thus positive.
Since  $ y_{1,5} = \psi(y)$,
positivity for $L_{y}$ cannot be
derived from the existence of a representing measure, since
\cite[Theorem 2.4]{F08}
shows that $y$ has no representing measure. Moreover,
positivity for $L_{y}$ cannot be established from the positivity
of $M$
via sums-of-squares arguments
because, by Hilbert's theorem, there exist
degree 6 bivariate polynomials
that are everywhere nonnegative but are not sums of squares.
To prove that $L_{y}$ is $K$-positive, we employ
a  sequence of approximate representing measures.
Since $J\succ 0$, $t\equiv  \Delta > U^{T}N^{-1}U$. Thus, there
exists $\delta > 0$ such that if
we replace $s$ ($= \psi(y)$) by $s+\frac{1}{m}$  (with $\frac{1}{m}< \delta$),
then the resulting moment matrix, $M_{3}(y^{(m)})$, remains
positive, with $rank~ M_{3}(y^{(m)}) = 9$ and $X_{2} = X_{1}^{3}$. Since
$\psi(y^{(m)})$ is independent of $y_{15}[y^{(m)}]$ and
$y_{06}[y^{(m)}]$, we have $\psi(y^{(m)}) = \psi(y) = s <
    s+\frac{1}{m} =               y_{15}[y^{(m)}]$.
It now follows from \cite[Theorem 2.4]{F08} that $y^{(m)}$
has a $K$-representing measure,
whence $L_{y^{(m)}}$ is $K$-positive. Since $\| y^{(m)} - y \| = \frac{1}{m}
\longrightarrow 0$, we conclude that $L_{y}$ is $K$-positive, and
thus positive.
\qed
\end{example}

\begin{remark}
We have previously noted an example of \cite[Section~4]{CF98}
(based on a construction of Schm\"{u}dgen \cite{Sch})
which illustrates
a case where, with $n=2$, $M_{3}(y)\succ 0$ but $L_{y}$ is not positive.
Example 2.5 shows that if $M_{3}(y)\succeq 0$ and $L_{y}$ is positive,
$y$ need not have a representing measure. Whether this can happen
with $M_{3}(y)\succ 0$ is the content of Question 1.2.
\end{remark}

Now we introduce a variety associated to
$L_{y}$ that provides a finer tool than $\mathcal{V}(M_{d}(y))$
for studying issues related to Question 1.2. For a moment sequence $y$
of degree $2d$, we define the variety of $L_{y}$ by
$$     V(L_{y}) := \bigcap_{p\in \mathcal{P}_{2d},~
            p|\mathcal{V}(M_{d}(y))\ge 0, ~ L_{y}(p)=0}  \mathcal{Z}(p).$$
\begin{prop} \label{sup(mu)<=V(Ly)}
If $y$ has a representing measure $\mu$, then $supp~\mu \subseteq V(L_{y})$.
\end{prop}
\begin{proof}
Suppose there exists $u\in supp~\mu$ such that $u\not \in V(L_{y})$.
Then there exists some $p\in \mathcal{P}_{2d}$, such that
$ p|\mathcal{V}(M_{d}(y))\ge 0$ and $L_{y}(p)=0 $, but
$p(u) \ne 0$. Since $supp~\mu \subseteq  \mathcal{V}(M_{d}(y))$,
we have $p|supp~\mu \ge 0$, and hence $p(u)>0$.
Thus, it follows that $L_{y}(p) = \int_{supp~\mu} p(t)d\mu(t) > 0$,
which contradicts $L_{y}(p)=0$.
\end{proof}

\begin{prop}
For each truncated moment sequence $y$,
$V(L_{y}) \subseteq  \mathcal{V}(M_{d}(y)).$
\end{prop}
\begin{proof}
Let $p$ be an arbitrary polynomial such that
$p\in \mathcal{P}_{d}$ and $p(X)=0$ in the column space of $M_{d}(y)$.
Then $L_{y}(p^{2}) = \langle M_{d}(y)\hat{p},~\hat{p} \rangle
=0$. Since $p^{2}|  \mathcal{V}(M_{d}(y)) \ge 0$, it follows that
$V(L_{y}) \subseteq \mathcal{Z}(p^{2}) =   \mathcal{Z}(p)$.
By definition of $\mathcal{V}(M_{d}(y))$ in \reff{def:momvar},
the result is proved.
\end{proof}

In view of Proposition 2.8, the following result refines the
necessary condition $rank~M_{d}(y) \le card~ \mathcal{V}(M_{d}(y))$.

\begin{cor}
If $y$ has a representing measure, then $rank~M_{d}(y) \le card~V(L_{y})$.
\end{cor}
\begin{proof}
Let $\mu$ be a representing measure for $y$. Then $rank~M_{d}(y)
\le card~supp~\mu$ (see relation~\reff{rank<=card} in Section~1),
and the result follows from Proposition~\ref{sup(mu)<=V(Ly)}.
\end{proof}

 We conclude this section with an example which shows that
 $V(L_{y})$ may be a proper subset of $ \mathcal{V}(M_{d}(y))$
(in a case where $y$ has a representing measure).
\begin{example}
For $n=2$, $d=3$, consider the moment matrix
$$ M_{3}(y):=
\bbm
8 & 0 & 0 & 6 & 0 & 6   & 0 & 0 & 0 & 0      \\
0 & 6 & 0 & 0 & 0 & 0   & 6 & 0 & 4 & 0      \\
0 & 0 & 6 & 0 & 0 & 0   & 0 & 4 & 0 & 6      \\
6 & 0 & 0 & 6 & 0 & 4  & 0 & 0 & 0 & 0      \\
0 & 0 & 0 & 0 & 4 & 0 & 0 & 0 & 0 & 0      \\
6 & 0 & 0 & 4 & 0 & 6 & 0 & 0 & 0 & 0      \\
0 & 6 & 0 & 0 & 0 & 0   & 6 & 0 & 4 & 0      \\
0 & 0 & 4 & 0 & 0 & 0  & 0 & 4 & 0 & 4      \\
0 & 4 & 0 & 0 & 0 & 0   & 4 & 0 & 4 & 0      \\
0 & 0 & 6 & 0 & 0 & 0  & 0 & 4 & 0 & 6
\ebm .
$$
A calculation shows that $M_{3}(y)\succeq 0$, with
$rank~M_{3}(y) = 8$.  $ \mathcal{V}(M_{3}(y))$ is
determined by the column relations $X_{1} = X_{1}^{3}$
and $X_{2}= X_{2}^{3}$, and thus consists of the 9 points
$u_{1}= (0,0)$,  $u_{2}= (0,1)$,
$u_{3}= (0,-1)$,  $u_{4}= (-1,0)$,
$u_{5}= (-1,1)$,  $u_{6}= (-1,-1)$,
$u_{7}= (1,0)$,  $u_{8}= (1,1)$,
$u_{9}= (1,-1)$.
Observe that $y$ has the $8$-atomic representing measure
$\mu :=     \sum_{i=2}^{9} \delta_{u_{i}}$,
and we will show that $V(L_{y})= supp~\mu$, so that $V(L_{y})$ is
a proper subset of $ \mathcal{V}(M_{3}(y))$.
To see this, we consider the dehomogenized  Robinson polynomial,
$$ r(x_{1},x_{2}) =   x_{1}^{6} + x_{2}^{6} -x_{1}^{4}x_{2}^{2}
- x_{1}^{2}x_{2}^{4} - x_{1}^{4}
- x_{2}^{4}-x_{1}^{2}-x_{2}^{2}  + 3 x_{1}^{2}x_{2}^{2} + 1.$$
It is known that $r(x_{1},x_{2})$ is nonnegative on $\re^{2}$
and has exactly 8 zeros in the
affine plane, namely the points in $supp~\mu$ (cf. \cite{Rez07}).
A calculation shows that $L_{y}(r) = 0$, so
$V(L_{y}) \subseteq \mathcal{Z}(r) = supp~\mu \subseteq V(L_{y})$
(by Proposition~\ref{sup(mu)<=V(Ly)}),
so $V(L_{y}) = supp~\mu$ and thus
$V(L_{y})$ is a proper subset of $ \mathcal{V}(M_{3}(y))$.

It is known that $r(x_{1},x_{2})$ is not a sum of squares
(cf. \cite{Rez07});
to see this using variety methods,
suppose to the contrary that $r=\sum_{i} r_{i}^2$, with
each $r_{i}\in \mathcal{P}_{3}$. Then
$supp~\mu = \mathcal{Z}(r) = \bigcap_{i}  \mathcal{Z}(r_{i})$,
whence $supp~\mu \subseteq  \mathcal{Z}(r_{i})$ for each $i$.
It now follows from \cite{CF96} that for each $i$,
$r_{i}(X_{1},X_{2}) = 0$ in the column space
of $M_{3}(y)$. Thus, we have $ \mathcal{V}(M_{3}(y))
\subseteq \bigcap_{i}  \mathcal{Z}(r_{i}) = supp~\mu$,
a contradiction. This example also illustrates a moment sequence
 $y$ with a $rank~M_{d}(y)$-atomic representing measure
 and $rank~M_{d}(y) < card~ \mathcal{V}(M_{3}(y))< +\infty$;
the first such example appears in \cite{F082}.
\qed
\end{example}

\section{Solution of the bivariate quartic moment problem}
\setcounter{equation}{0}

Throughout this section, we consider bivariate quartic moment problems,
that is, $n=2$ and the degree $2d=4$.
Let $y \in \mc{M}_{2,4}$ be a truncated moment sequence of degree 4,
which is associated with the second order moment matrix
\[
M_2(y) :=
\bbm
y_{00} & y_{10} & y_{01} & y_{20} & y_{11} & y_{02} \\
y_{10} & y_{20} & y_{11} & y_{30} & y_{21} & y_{12} \\
y_{01} & y_{11} & y_{02} & y_{21} & y_{12} & y_{03} \\
y_{20} & y_{30} & y_{21} & y_{40} & y_{31} & y_{22} \\
y_{11} & y_{21} & y_{12} & y_{31} & y_{22} & y_{13} \\
y_{02} & y_{12} & y_{03} & y_{22} & y_{13} & y_{04}
\ebm .
\]
As noted in Introduction (cf. (1.3)), if $M_{2}(y)$ is singular,
then $y$ has a representing measure if and only if
$M_{2}(y)$ is positive semidefinite, recursively generated,
and $rank~M_{2}(y)\le card~\mathcal{V}(M_{2}(y))$.
\begin{example}
Consider
\[
M_2(y)=
\bbm
8 & 0 & 0 & 4 & 0 & 4 \\
0 & 4 & 0 & 2 & 0 & -2 \\
0 & 0 & 4 & 0 & -2 & 0 \\
4 & 2 & 0 & 11 & 0 & a \\
0 & 0 & -2 & 0 & a & 0 \\
4 & -2 & 0 & a & 0 & b \\
\ebm .
\]
With $a=1$ and $b=3$,
$M_{2}(y)$ is positive and recursively generated,
with column relations $X_{1}=1-2X_{2}^{2}$ and $X_{2}=-2X_{1}X_{2}$,
and $rank~M_{2}(y) = 4$. A calculation shows that
$x_1=1-2x_2^{2}$ and $x_2=-2x_1x_2$ have only 3 common zeros,
so $3=card~\mathcal{V}(M_{2}(y)) <  rank~M_{2}(y) = 4 $,
whence (1.3) implies that $y$ has no representing measure.
We will show below how to approximate $y$ with truncated
moment sequences having representing measures.\qed
\end{example}

For the case when $M_{2}(y)\succ 0$, it has been an open question
as to whether $y$ admits a representing measure. The aim of this
section is to give an affirmative answer to this question.
We begin, however, by showing that when $M_{2}(y)$ is merely positive
semidefinite, then $y$ admits approximate representing measures.

\begin{thm} \label{thm:clos-rep(2,4)}
If $y \in \mc{M}_{2,4}$ and $M_2(y)\succeq 0$, then $y \in \overline{\mc{R}_{2,4}}$.
\end{thm}
\begin{proof}
Let $y \in \mc{M}_{2,4}$ be such that $M_2(y)\succeq 0$.
To show $y \in \overline{\mc{R}_{2,4}}$,
by Theorem~\ref{thm:pos=>clos-rep},
it suffices to show that
the Riesz functional $L_y$ is positive.
If a polynomial $p(x) \in \mc{P}_4$ is nonnegative on the plane $\re^2$,
then by Hilbert's theorem it
must be a sum of squares, so
there exist bivariate quadratic
polynomials $q_1(x), \ldots, q_m(x)$, $deg~q_{i} \le 2$ $(1\le i \le m)$,
 such that
\[
p(x) = q_1(x)^2 + \cdots + q_m(x)^2.
\]
Hence, since $M_{2}(y)\succeq 0$, we have
\[
L_y(p) = L_y(q_1^2) + \cdots + L_y(q_m^2) =
\langle M_2(y) \hat{q_1}, \hat{q_1} \rangle
+ \cdots + \langle M_2(y) \hat{q_m}, \hat{q_m} \rangle \geq 0,
\]
so $L_{y}$ is positive.
It now follows from
Theorem~\ref{thm:pos=>clos-rep} that $y \in \overline{\mc{R}_{2,4}}$.
\end{proof}

Note that if $M_{2}(y)$ is positive and singular, and
$y$ does not have a representing measure, then $L_{y}$ is
positive, but not strictly positive. Indeed, positivity
follows from Theorem 3.2. Since $M_{2}(y)$ is singular,
there exists $p\in \mathcal{P}_{2}$, $p \not \equiv 0$,
such that $M_{2}(y)\hat{p} = 0$; then $p^{2}\ge 0$ and
$L_{y}(p^{2}) = \langle M_{2}(y)\hat{p},\hat{p} \rangle = 0$,
so $L_{y}$ is not strictly positive.

We now turn to the positive definite case.
The following result provides an affirmative answer to Question 1.2
for the case $n = d = 2$, $K = \re^2$.

\begin{thm}  \label{thm:n=2deg=4}
If  $M_2(y) \succ 0$,
then $y$  has a representing measure.
\end{thm}
\begin{proof}
Clearly $\re^2$ is a determining set.
By Theorem~\ref{thm:strpos=>rep}, it suffices to show that $L_{y}$ is strictly
positive. Proceeding as in the previous proof, if $p\in \mc{P}_{4}$
is nonnegative on $\re^2$ and not identically zero, then $p$ is of the form
$p(x) = q_1(x)^2 + \cdots + q_m(x)^2$,
with $deg~q_{i}\le 2$ $(1\le i \le m)$ and every $q_{i}\not \equiv 0$.
Since $M_{2}(y)\succ 0$, we have
$L_y(p) = L_y(q_1^2) + \cdots + L_y(q_m^2) =
\langle M_2(y) \hat{q_1}, \hat{q_1} \rangle
+ \cdots + \langle M_2(y) \hat{q_m}, \hat{q_m} \rangle >0$,
and the result follows.
\end{proof}

\begin{remark}
Theorem \ref{thm:n=2deg=4} shows that if $n=2$ and $M_{2}(y)\succ 0$, then $y$
has a representing measure, whence \cite{BT} implies that $y$
has a representing measure $\mu$ with $card~supp~\mu
\le dim~\mc{P}_{4} = 15$.
We do not have a better upper
bound for the size of the support,
and it remains an
open problem as to whether, in this case, $M_{2}(y)$ actually
has a flat extension $M_{3}(\tilde{y})$, with a corresponding
$6$-atomic representing measure for $y$. In the case
when $n=2$ and $M_{2}(y)$ is positive semidefinite and singular,
$y$ has a representing measure if and only if the conditions of
(1.3) hold, and in this case, the results of \cite{CF052} show that
either $M_{2}(d)$ has a flat extension $M_{3}(\tilde{y})$,
or $M_{2}(y)$ admits a positive extension  $M_{3}(\tilde{y})$
satisfying $rank~M_{3}(\tilde{y}) = 1 + rank~M_{2}(y)$,
and $M_{3}(\tilde{y})$ has a flat extension.
This leads to our next question (cf. \cite{CF02,LL}).
\end{remark}

\begin{quest} \label{quest3.5}
If $y\in \mc{M}_{2,4}$ and $M_2(y) \succ 0$,
does $M_2(y)$ have a flat extension?
Does $y$ have an extension $\tilde{y} \in \mc{M}_{2,6}$
such that $M_3(\tilde{y})$ is positive and has a flat extension?
\end{quest}

We next present
two examples which illustrate
Theorem 3.2 in cases where $y$ has no representing measure.

\begin{exm}
Consider the moment sequence $y\in \mc{M}_{2,4}$ such that
\[
M_2(y) =
\bbm
1 & 1 & 1 & 1 & 1 & 1 \\
1 & 1 & 1 & 1 & 1 & 1 \\
1 & 1 & 1 & 1 & 1 & 1 \\
1 & 1 & 1 & 2 & 2 & 2 \\
1 & 1 & 1 & 2 & 2 & 2 \\
1 & 1 & 1 & 2 & 2 & 2
\ebm .
\]
Clearly, $M_{2}(y)\succeq 0$.
Since $X_{1}=1$ but $X_{1}^{2} \not = X_{1}$, $M_{2}(y)$ is not recursively
generated, so $y$ has no representing measure.
However, by Theorem~\ref{thm:clos-rep(2,4)},
$y$ lies in the closure of moment sequences having representing measures.
To see this explicitly,
define the moment sequence $y(\eps)$  via the moment matrix  $M_{2}(y(\eps)):=$
\[
 \bbm 1 & 1 + \eps^{3/4}-\eps & 1+\eps^{3/4}-\eps & 1+\eps^{1/2}-\eps & 1+\eps^{1/2}-\eps & 1+\eps^{1/2}-\eps \\
          1+ \eps^{3/4}-\eps  & 1+ \eps^{1/2}-\eps & 1+ \eps^{1/2}-\eps & 1+\eps^{1/4}-\eps & 1+\eps^{1/4}-\eps &  1+\eps^{1/4}-\eps \\
 1+ \eps^{3/4}-\eps  & 1 + \eps^{1/2}-\eps  & 1 + \eps^{1/2}-\eps  & 1+\eps^{1/4}-\eps & 1+\eps^{1/4}-\eps & 1+\eps^{1/4}-\eps   \\
 1 + \eps^{1/2}-\eps & 1+\eps^{1/4}-\eps & 1+\eps^{1/4}-\eps & 2-\eps & 2- \eps & 2- \eps\\
 1+ \eps^{1/2}-\eps  & 1+ \eps^{1/4}-\eps & 1+ \eps^{1/4}-\eps & 2-\eps & 2-\eps &  2-\eps \\
 1+ \eps^{1/2}-\eps  & 1 + \eps^{1/4}-\eps  & 1 + \eps^{1/4}-\eps  & 2-\eps & 2-\eps & 2-\eps
   \ebm .\]
A calculation shows that $y(\eps)$ has the $2$-atomic representing measure
\[
(1-\eps)\delta_{(1,1)} + \eps \delta_{(\eps^{-1/4},\eps^{-1/4})},
\]
and
obviously $y(\eps) \to y$ as $\eps \to 0$.
\qed
\end{exm}

\begin{exm}
Let us return to Example 3.1. With $a=1$ and $b=3$, $M_{2}(y)$ is
positive semidefinite, so although $y$ has no representing measure,
Theorem 3.2 implies that
$y$ can be
approximated by moment sequences having measures. One way to do this
is to replace $b=3$ by $b=3+\frac{1}{m}$. The resulting moment sequence
$y^{(m)}$ satisfies      $M_{2}(y^{(m)})\succeq 0$ and
    $M_{2}(y^{(m)})$ is recusrsively generated. Further,
$\mathcal{V}(M_{2}(y^{(m)}))
= \{(x_1,x_2): x_2 = -2x_1x_2\}$, and since the variety is infinite,
(1.3) implies that $y^{(m)}$ has a representing measure.
Following \cite[Proposition 3.6]{CF052},
a calculation shows
 that although $M_{2}(y^{(m)})$ admits no flat extension
 $M_{3}(\widetilde{y^{(m)}})$
 (so $y^{(m)}$ has no $5$-atomic representing measure),
 $M_{2}(y^{(m)})$ does admit a positive extension
  $M_{3}(\widetilde{y^{(m)}})$, with $rank  ~M_{3}(\widetilde{y^{(m)}}) = 6$,
such that  $M_{3}(\widetilde{y^{(m)}})$ has a flat extension
  $M_{4}(\widetilde{\widetilde{y^{(m)}}})$. Thus, $y^{(m)}$ has a $6$-atomic
  representing measure.

 Another approach is to replace $a=1$ by $a= 1+\frac{1}{m}$
and $b=3$ by $b= 3 + \frac{1}{4m^{2}}$. Then the resulting
moment sequence $y^{(m)}$ has      $M_{2}(y^{(m)})\succ 0$,
so $y^{(m)}$ has a representing measure by Theorem~\ref{thm:n=2deg=4}.
Indeed, a {\it{Mathematica}} calculation shows that
with $ \widetilde{y^{(m)}}_{4,1}=\widetilde{y^{(m)}}_{2,3}=
 \widetilde{y^{(m)}}_{1,4}=   \widetilde{y^{(m)}}_{0,5}=0$,
 $M_{2}(y^{(m)})$ admits two distinct flat extensions
 $M_{3}(\widetilde{y^{(m)}})$
 (and corresponding $6$-atomic representing measures
 for $y^{(m)}$). \qed
\end{exm}

We conclude this section with an application of
Theorem~\ref{thm:n=2deg=4} to a solution to the
bivariate cubic moment problem, with $y$ of the form
$$y = \
\{y_{00},~y_{10},~y_{01},~       y_{20},~y_{11},~y_{02},
~  y_{30},~       y_{21},~y_{12},~y_{03}\},$$ with $y_{00}>0$.
To such a sequence we may associate $M_{1}(y)$ and
the block
\[
B(2) :=
\bbm
y_{20} & y_{11} & y_{02} \\
y_{30} & y_{21} & y_{12} \\
y_{21} & y_{12} & y_{03}
\ebm.
\]

\begin{thm}  \label{thm:deg=3}
Suppose $y\in \mathcal{M}_{2,3}$. If $y$ has a representing measure,
then $M_{1}(y)\succeq 0 $. Conversely, suppose $M_{1}(y)\succeq 0$.
\newline i) If $M_{1}(y)\succ 0$, then $y$ has a representing measure.
\newline ii) If $rank~M_{1}(y) = 2$, then $y$ has a representing measure
if and only if $Ran~B(2)\subseteq Ran~M_{1}(y)$ and
$\bbm
M_{1}(y) & B(2)
\ebm$ is recursively generated.
\newline iii) If $rank~M_{1}(y) = 1$, then $y$ has a representing
measure if and only if   $Ran~B(2)\subseteq Ran~M_{1}(y)$.
\end{thm}
\begin{proof}
Since a representing measure for $y$ is, in particular,
 a representing measure for
$M_{1}(y)$, the necessity of the condition $M_{1}(y)\succeq 0$ is clear.
Conversely, suppose   $M_{1}(y)\succeq 0$.
For i), if $M_{1}(y)\succ 0$, then it is
not difficult to see that $M_{1}(y)$ admits a positive
definite moment matrix extension $M_{2}$, of the form
\[
M_{2} \equiv
\bbm
M_{1}(y) & B(2) \\
B(2)^{T} & C(2)
\ebm,
\]
where
\[
C(2) =
\bbm
y_{40} & y_{31} & y_{22} \\
y_{31} & y_{22} & y_{13} \\
y_{22} & y_{13} & y_{04}
\ebm.
\]
Indeed, by choosing $y_{40}$, $y_{22}$, and $y_{04}$
successively, and sufficiently large, we can insure
that $C(2)\succ B(2)^{T}M_{1}(y)^{-1}B(2)$.
By   Theorem~\ref{thm:n=2deg=4}, $M_{2}$ then has a representing measure,
which is obviously a representing measure for $y$.

Suppose next that $y$ has a representing measure. It follows from
\cite{BT} that $y$ has a finitely atomic representing measure $\mu$,
and thus $M_{2}[\mu]$ is a positive semidefinite and
recursively generated extension of
$M_{1}(y)$.
In particular, we must have $Ran~B(2)\subseteq Ran~M_{1}(y)$
and
$\bbm
M_{1}(y) & B(2)
\ebm$ must be recursively generated. Now suppose that these
conditions hold and that $rank~M_{1}(y)=2$. Since $y_{00}>0$,
we may assume without loss of generality that there exist
scalars $\alpha$ and $\beta$ so that in the column space of
$M_{1}(y)$
we have a column dependence relation
\begin{equation}
 X_{2} = \alpha 1 + \beta X_{1}.
\end{equation}
Since
$\bbm
M_{1}(y) & B(2)
\ebm$ is recursively generated, we then have  the column relations
\begin{equation}
 X_{1}X_{2} = \alpha X_{1} + \beta X_{1}^{2},
\end{equation}
\begin{equation}
 X_{2}^{2} = \alpha X_{2} + \beta X_{1}X_{2}.
\end{equation}
Since $Ran~B(2)\subseteq Ran~M_{1}(y)$, there is a matrix $W$
such that $B(2) = M_{1}(y)W$, and we may thus define a positive,
rank-preserving
extension $M$ of $M_{1}(y)$ by
\[
M :=
\bbm
M_{1}(y) & B(2) \\
B(2)^{T} & C
\ebm,
\]
where $C := B(2)^TW$ ($= W^{T}M_{1}(y)W$).
It is straightforward to check that
the columns of $M$ satisfy (3.1)-(3.3), from which
it also follows that $M$ has the form of a moment matrix $M_{2}$.
Thus $M$ is a flat, positive moment matrix extension of $M_{1}(y)$,
whence \cite{CF05} implies the existence of a representing measure for
$M$, and thus for $y$.

The proof of iii) is similar to the proof of ii), but simpler.
It is straightforward to
check that if $rank~M_{1}(y) = 1$ and $Ran~B(2)\subseteq Ran~M_{1}(y)$,
then the dependence relations in the columns of $M_{1}(y)$ propagate
recursively so as to define a rank one (flat, positive) moment matrix
extension $M_{2}(y)$ of $M_1(y)$. The result follows as above.
\end{proof}

\section{Quadratic moment problems}

Let $y=(y_\af)_{\af\in \Z_{+}^n: |\af| \leq 2}$ be a quadratic moment sequence
such that $M_1(y) \succeq 0$.
Does $y$ have a representing measure?
For this question, we may assume without loss of generality that
$y_0 = 1$ and we may write $M_1(y)$ as
\[
M_1(y) = \bbm 1 & v_1^T \\ v_1 &  U \ebm,
\]
where $v_1 \in \re^{n}$.
Since $M_1(y) \succeq 0$, then $U-v_1v_1^T \succeq 0$, so the
Spectral Theorem implies that
 there exist vectors $v_2, \ldots, v_r$ in $\re^{n}$ such that
\[
U = v_1v_1^T + v_2v_2^T + \ldots + v_rv_r^T.
\]
A calculation now shows that we have
\[
M_1(y) = \frac{1}{r-1} \sum_{i=2}^r
\left(
\bbm 1 \\ v_1 \ebm \bbm 1 \\ v_1 \ebm^T +
(r-1) \bbm 0 \\ v_i \ebm \bbm 0 \\ v_i \ebm^T
\right).
\]
For $i=2, \ldots, r$, let $u_i^+ = v_1 + \sqrt{r-1} v_i,  u_i^- = v_1 - \sqrt{r-1} v_i$.
Then we have the representation
\[
M_1(y) = \frac{1}{2(r-1)} \sum_{i=2}^r
\left(
\bbm 1 \\ u_i^+ \ebm \bbm 1 \\ u_i^+ \ebm^T +
\bbm 1 \\ u_i^- \ebm \bbm 1 \\ u_i^- \ebm^T
\right),
\]
and hence we know $y$ has a $(2r-2)$-atomic representing measure
$$\mu = \sum_{i=2}^{r}  \frac{1}{2(r-1)} (\delta_{u_{i}^{+}} +
         \delta_{u_{i}^{-}} ).$$
In the sequel, we will show that $y$ actually has a
$rank~M_{1}(y)$-atomic representing measure (equivalently,
$M_{1}(y)$ admits a flat extension $M_{2}(\tilde{y})$).

Now we turn to the quadratic truncated moment problem
on an algebraic set $E(q):=\{x\in \re^n: q(x) = 0\}$
or a semialgebraic set $S(q):=\{x\in \re^n: q(x) \geq 0\}$,
where $q(x)$ is a quadratic polynomial in $x$.
If $y \in \mc{M}_{n,2}$ has a representing measure
supported in $E(q)$, it is necessary that
\[
M_1(y) \succeq 0, \quad L_y(q) = 0.
\]
Is the above also sufficient
for $y$ to have a representing measure
supported in $E(q)$?
If $y \in \mc{M}_{n,2}$ has a representing measure
supported in $S(q)$, it is necessary that
\[
M_1(y) \succeq 0, \quad L_y(q) \geq 0.
\]
Is the above also sufficient
for $y$ to have a representing measure
supported in $S(q)$?
These questions will be answered affirmatively under certain suitable conditions.

Throughout this section, we will employ a well-known connection between
nonnegative polynomials and
positive semidefinite real symmetric matrices (cf. \cite{Lau}),
which we apply in the case of quadratic polynomials.
Let $p(x) = \sum_{ \af \in\Z_{+}^n: |\af| \le 2} p_\af x^{\af}$.
Let $y_{p}$ denote the degree 2 moment sequence whose moment corresponding
to a monomial of degree 1, or to a monomial of the form
$x^\af = x_{i}x_{j}~ (i\not= j)$, is
 $p_\alpha/2$, and whose
moment corresponding to a monomial of degree 0, or
of the form $x^\af = x_{i}^2$, is
$p_\alpha$. A calculation shows that
\begin{equation} \label{eq:p-Qrep}
p(x) = [x]_{1}^{T}M_{1}(y_{p})[x]_{1}.
\end{equation}
 From this it follows immediately
that $p(x)$ is nonnegative on $\re^n$ if and only if there exists
a matrix $P$ such  that $P=P^{T}$, $P\succeq 0$, and
\begin{equation} \label{p=[]1^TP[]1}
  p(x) = [x]_{1}^{T}P[x]_{1}~ (x\in \re^n).
\end{equation}
In the case when $p(x)$ is a homogeneous quadratic, by compressing
$M_{1}(y_{p})$ to the rows and columns indexed by the variables $x_{i}$,
and similarly for $[x]_{1}$, we see that $p(x)$
admits a representation of the form
\begin{equation} \label{p(x)=x^TPx}
  p(x) = x^{T}Px~ (x\in \re^n),
\end{equation}
where $P=P^{T}$; further,
$p(x)$ is nonnegative on
$\re^n$ if and only if
$P\succeq 0$.

In the sequel, for $m\times m$ real matrices  $R\equiv (r_{ij})$ and
$S\equiv (s_{ij})$, we denote by
$R \bullet S$ the Frobenius inner product, defined by
$R \bullet S = Trace(RS^T) = \sum_{1\le i,j \le m} r_{ij}s_{ij}$.
A calculation shows that if $p$ has a representation as in \reff{p=[]1^TP[]1}
and $y$ is a quadratic moment sequence, then
\begin{equation} \label{Lyp=PdotM1(y)}
L_{y}(p) = P\bullet M_{1}(y).
\end{equation}
If $R=R^{T}\succeq 0$ and $S=S^{T}\succeq 0$,
then $R=LL^T$ and $S=MM^T$, and thus
\begin{align*}
R\bullet S &= Trace(LL^TMM^T) =  Trace(M^TLL^TM) \\
& =  Trace(M^TL(M^TL)^T) = (M^TL) \bullet (M^TL) \ge 0.
\end{align*}
It now follows that
\begin{equation} \label{PSDdot>=0}
if~ R=R^{T}\succeq 0~and~S=S^{T}\succeq 0,~then~R\bullet S\ge 0.
\end{equation}

\subsection{Quadratic polynomials nonnegative on quadratic sets}

A useful tool in quadratic moment theory, which we will employ
repeatedly,
is the following matrix decomposition
developed by Sturm and Zhang \cite{SZ}.
In the sequel, let $\mathbb{M}_{m}$ denotes the space of real
$m \times m$ matrices, endowed with the norm induced by the
Frobenius inner product.

\begin{prop}[Corollary~4 \cite{SZ}] \label{prop:SZ-decomp}
Let $Q \in \mathbb{M}_{m} $ be a symmetric matrix.
If $X \in  \mathbb{M}_{m} $ is symmetric positive semidefinite
and has rank $r$, then
there exist nonzero vectors $u_1,\ldots, u_r \in \re^m$ such that
\[
X = u_1u_1^T+\cdots+u_ru_r^T, \quad
u_1^TQu_1 = \cdots = u_r^TQu_r = \frac{Q \bullet X}{r}.
\]
\end{prop}

We will also utilize the following representation of
quadratic polynomials that are nonnegative on $S(q)$.

\begin{prop}[S-Lemma, Yakubovich (1971), \cite{Yak77}] \label{prop:S-lema}
Let $f(x),q(x)$ be two quadratic polynomials in $x$.
Suppose there exists $\xi\in \re^n$ such that $q(\xi)>0$.
Then
$f(x) \geq 0$ for all $x\in S(q)$ if and only if there exists $t\geq 0$ such that
\[
f(x) - t q(x) \geq 0, \quad \forall \,x \in \re^n.
\]
\end{prop}

When $f(x)$ and $g(x)$ are homogeneous and quadratic,
if $f(x)$ is nonnegative
on the algebraic set $E(q)=\{x\in\re^n: q(x)=0\}$, then
a certificate like that provided by S-Lemma holds,
but without requiring $t\geq 0$,
as pointed out by Luo, Sturm and Zhang \cite{LSZ}.
However, we are not able  to find a complete proof
from \cite{LSZ} and the references therein.
Moreover, this result can also be generalized to
the case when $f(x)$ and $g(x)$ are non-homogeneous.
So here we summarize these results and
include a proof for completeness.

\begin{prop}\label{prop:S-lm-E(q)}
Let $f(x),q(x)$ be two quadratic polynomials in $x$, and assume
$E(q) \not= \emptyset$.
Suppose $f(x) \geq 0$ for all $x\in E(q)$, and suppose
there exist $\xi,\zeta\in \re^n$ such that $q(\xi)>0>q(\zeta)$.
Then
there exists $t\in\re$ such that
\[
f(x) - t q(x) \geq 0, \quad \forall \,x \in \re^n.
\]
\end{prop}
\begin{proof}
{\bf Step~1}\quad
Consider first the case when both
$f$ and $g$ are homogeneous quadratics.  From \reff{p(x)=x^TPx}, we may
 write $f(x)=x^TFx$ and $q(x)=x^TQx$ for symmetric matrices
 $F,Q \in \mathbb{M}_{n}$.
In the sequel we view  $\mathbb{M}_{n}$ as a locally convex normed
real vector space, with norm induced by the Frobenius inner product.
By finite dimensionality,
each linear functional on   $\mathbb{M}_{n}$ is of the
form $F \longrightarrow F \bullet X$ for some $X\in  \mathbb{M}_{n}$.
Let $\mc{E} = \{ S + t Q: S^T=S\succeq 0, t \in \re \}$.
Obviously $\mc{E}$ is a convex set, and we claim that
$\mc{E}$ is also closed.
To see this,  let $\{A_k \equiv S_k + t_k Q\} \subset \mc{E}$
be sequence such that $A_k \to A$. Note that every $S_k\succeq 0$ and thus
\[
\xi^TA_k\xi = \xi^TS_k\xi + t_k \xi^T Q \xi \geq t_k \xi^T Q \xi, \quad
\zeta^TA_k\zeta = \zeta^TS_k\zeta + t_k \zeta^T Q \zeta \geq t_k \zeta^T Q \zeta.
\]
From this, and the hypothesis
$\xi^TQ\xi = q(\xi)>0>q(\zeta)=\zeta^TQ\zeta$,
it follows that
\[
\frac{\zeta^TA_k\zeta}{q(\zeta)}  \leq  t_k \leq \frac{\xi^TA_k\xi}{q(\xi)}.
\]
Since $\{A_{k}\}$ is bounded, $\{\zeta^TA_k\zeta\}$ and $\{\xi^TA_k\xi\}$ are also bounded,
whence the sequence $\{t_k\}$ is bounded too. Thus $\{S_k\}$ is also bounded,
so we may assume $S_k \to S_* \succeq 0$ and $t_k \to t_*$, whence
$A_k \to A = S_* + t_* Q \in \mc{E}$.

Now we show that $F$ belongs to the closed convex set $\mc{E}$.
 Suppose to the contrary that  $F \not\in \mc{E}$.
 It follows from a version of the
 Hahn-Banach Theorem \cite[Proposition 14.15]{BP}
that there exist a nonzero symmetric matrix $X$ and a scalar $\eta$ such that
\[
F\bullet X < \eta, \quad
(S+t Q) \bullet X \geq \eta, \, \forall \,\, S^T=S \succeq 0,~ t\in \re.
\]
By choosing $S=0$, we see that $Q \bullet X =0$. Thus $S\bullet X\ge \eta
~~
\forall \, S^T=S\succeq 0$, whence  $X \succeq 0$.
The preceding implies that
\[
Q \bullet X =0, \quad  X \succeq 0, \quad \eta \leq 0.
\]
Then, by Proposition~\ref{prop:SZ-decomp},
there exist vectors $u_1, \ldots, u_r$ such that
\[
X=u_1u_1^T+\cdots+u_ru_r^T, \quad
u_i^TQu_i = 0\quad (1\le i \le r).
\]
From $ \sum_{i=1}^r u_i^T Fu_i = F\bullet X < \eta \leq 0$,
we see that at least one $u_i$ satisfies
\[
u_i^{T}Fu_i <0,~~ u_i^{T}Qu_i =0.
\]
Thus, $q(u_{i})=0$, but $f(u_{i})<0$,
which is a contradiction. So $F$ must belong to $\mc{E}$.
With $F= S + tQ$, for some $S^T=S\succeq 0$ and $t\in \re$,
 we have $f(x) = s(x) + tq(x)$
for some nonnegative quadratic $s(x)$
corresponding to $S$ via \reff{p(x)=x^TPx},
 so the result follows in this case.

\medskip
\noindent
{\bf Step~2}\quad
We next consider the case when at least one of $q$ and $f$ is non-homogeneous,
and without loss of generality in the following argument we may assume
both are non-homogeneous.
Since $E(q) \ne \emptyset$,
we may further assume that $q(0) = 0$
(for if $q(a) = 0$, we may replace $q$ and $f$ by $q(x+a)$ and $f(x+a)$).
Let $\tilde{q}(x_0,x) = x_0^2q(x/x_0)$ (resp. $\tilde{f}(x_0,x) = x_0^2f(x/x_0)$)
be the homogenization of $q(x)$ (resp. $f(x)$).
Denote $\tilde{x}^T := [x_0 \quad x^T]^T$,
and note that
\[
\tilde{f}(\tilde{x}) = x_0^2 f_0 + x_0 f_1(x) + f_2(x), \quad
\tilde{q}(\tilde{x}) =  x_0 q_1(x) + q_2(x),
\]
where every $f_i$ and $q_i$ are homogeneous of degree $i$.

Now we claim that
\be  \label{clm:QE-inf}
\tilde{f}(\tilde{x}) \geq 0, \quad \forall \, \tilde{x}: \, \tilde{q}(\tilde{x}) = 0.
\ee
From the hypothesis that $f(x)\ge 0$ whenvever $q(x)=0$,
(4.6) follows easily from the homogenization formulas when $x_0 \ne 0$.
For the case when $x_0=0$, we need to prove
\[
f_2(x) \geq 0, \quad \forall \, x: \, q_2(x) = 0.
\]
Let $u$ be an arbitrary point such that $q_2(u)=0$.
Consider the equation
\[
\tilde{q}(\eps, x) = \eps q_1(x) + q_2(x) = 0.
\]
If $q_1(u)=0$, then $q(\af u) =0$ for all real $\af$.
Thus $\af u \in E(q)$ and $f(\af u) \geq 0$ for all $\af$,
which implies $f_2(u) \geq 0$.
If $q_1(u) \ne 0$, then the rational function
\[
\eps(x)= - \frac{q_2(x)}{q_1(x)}
\]
is continuous in a neighborhood $\mc{O}_u$ of $u$.
Choose a sequence $\{u^{(i)}\} \subset \mc{O}_u$ such that
$q_2(u^{(i)}) \ne 0$ and $u^{(i)}\to u$.
Then $\eps(u^{(i)}) \ne 0$ and $\eps(u^{(i)})\to 0$.
 Since $\tilde{q}(\eps(u^{(i)}),u^{(i)})=
  \eps(u^{(i)})q_{1}(u^{(i)}) + q_{2}(u^{(i)}) = 0$,
  it follows that $q( \frac{u^{(i)}}{\eps(u^{(i)})} ) = 0$.
The hypothesis now implies that
$f( \frac{u^{(i)}}{\eps(u^{(i)})} ) \ge 0$, whence
$\tilde{f}(\eps(u^{(i)}),u^{(i)})\geq 0$.
Letting $i\to \infty$, we get
$f_2(u) = \tilde{f}(0,u)\geq 0$.
Therefore,  claim \reff{clm:QE-inf} is proved.
The existence of $\xi,\zeta\in \re^n$ such that $q(\xi)>0>q(\zeta)$ implies that
$\tilde{q}(1,\xi)>0>\tilde{q}(1,\zeta)$.
Now the homogeneous case can be applied to yield $t\in \re$ such that
$\tilde{f}(x_{0},x) -t\tilde{q}(x_{0},x)\ge 0 \quad \forall~ (x_{0},x)\in \re^{n+1}$,
and the result follows by setting $x_{0}=1$.
\end{proof}

In Proposition~\ref{prop:S-lm-E(q)},
if there do not exist $\xi,\zeta\in \re^n$ such that $q(\xi)>0>q(\zeta)$,
then the conclusion might fail.
For instance, for polynomials $f(x) = x_1x_2$ and $q(x) = -x_1^2$,
the summation $f(x) - t q(x)$ is never globally nonnegative
for any scalar $t$.
However, Proposition~\ref{prop:S-lm-E(q)}
can be weakened as follows.

\begin{prop} \label{prop:F+eps}
Let $f(x),q(x)$ be two quadratic polynomials. \\
(a)
If $S(q) \ne \emptyset$ and $f(x) \geq 0$ for all $x\in S(q)$,
then for any $\eps >0$ there exists $t\geq 0$ such that
\[
f(x) + \eps (1+\|x\|_2^2)  - t q(x) \geq 0, \quad \forall \,x \in \re^n.
\]
(b)
If $E(q) \ne \emptyset$ and $f(x) \geq 0$ for all $x\in E(q)$,
then for any $\eps >0$ there exists $t \in \re $ such that
\[
f(x) + \eps (1+\|x\|_2^2)  - t q(x) \geq 0, \quad \forall \,x \in \re^n.
\]
\end{prop}
\begin{proof}
As in \reff{eq:p-Qrep}, write $f(x)$ and $q(x)$ as
\[
f(x) = \bbm 1 \\ x \ebm^T
\underbrace{\bbm f_0 & f_1^T \\ f_1 & F_2 \ebm}_{F} \bbm 1 \\ x \ebm, \quad
q(x) = \bbm 1 \\ x \ebm^T
\underbrace{\bbm q_0 & q_1^T \\ q_1 & Q_2 \ebm}_{Q} \bbm 1 \\ x \ebm.
\]

(a) If there exists $\xi\in \re^n$ such that $q(\xi)>0$, then we are done by
Proposition~\ref{prop:S-lema}.
So we need only consider the case when $q(x)\leq 0$ for every $x\in \re^n$.
Since $S(q) \ne \emptyset$,
 without loss of generality we may assume that the origin belongs to $S(q)$,
which implies that $q_0 = 0$. Let
\[
\mc{E} = \{ S + t Q:\, S^T=S\succeq 0, t \geq 0 \}.
\]
Note that $\mc{E}$ is a convex set (but not necessarily closed).
We claim that for each $\eps >0$,
\[
F(\eps) :=  F + \eps I_{n+1} = \bbm f_0 +
\eps & f_1^T \\ f_1 & F_2 + \eps I_n \ebm \in \mc{E}.
\]
 Suppose to the contrary that $F(\eps) \not\in \mc{E}$ for some
 $\eps > 0$.
Then as in the proof of Proposition~\ref{prop:S-lm-E(q)},
there exist a nonzero symmetric matrix $X$ and a scalar $\eta$ such that
\[
F(\eps) \bullet X \leq \eta, \quad
(S+t Q)  \bullet X \geq \eta, \, \forall \, S\succeq 0, \, \forall \, t \geq 0.
\]
The above implies that
\[
Q \bullet X \geq 0, \quad  X \succeq 0, \quad \eta \leq 0.
\]
Then, by Proposition~\ref{prop:SZ-decomp},
there exist nonzero vectors $u_1, \ldots, u_r$ such that
\[
X=u_1 u_1^T+\cdots+u_r u_r^T, \quad
u_i^T Qu_i = \frac{Q \bullet X }{r} \geq  0 \quad (1\le i\le n).
\]
Write every $u_i$ as
\[
u_i = \bbm \tau_i \\ v_i \ebm, \quad \tau_i \in \re, ~~v_{i}\in \re^{n}.
\]
Order $u_i$ such that
$\tau_i \ne 0$ ($1\le i\le k$), and $\tau_{k+1} = \cdots =\tau_r = 0$
(the nonzero terms may be absent, or the zero terms may be absent).
For every $i=1,\ldots,k$ if $k>0$, we have
\[
 \tau_i^2 q(v_i/\tau_i) = u_i^TQu_i \geq 0.
\]
Thus every $v_i/\tau_i \in S(q)$
and hence $f( v_i / \tau_i) \geq 0$.
For every $i=k+1,\ldots,r$, we have
\[
 v_i^TQ_2v_i     =   u_i^TQu_i        \geq 0.
\]
Then we must have $v_i^TQ_2v_i =0$, because otherwise
$q(\af v_i) >0$ for $\af>0$ big enough
contradicts the assumption that $q(x)\leq 0$ for all $x\in\re^n$
at the beginning. So
\[
q(\af v_i) = 2 \af q_1^T v_i.
\]
Replacing $u_{i}$ by $-u_{i}$ if necessary, we may assume that
$q_1^Tv_i \geq 0$.
So $q(\af v_i ) \geq 0$ and $\af v_i \in S(q)$ for all $\af >0$.
Then $f(\af v_i)\geq 0$ for all $\af >0$, and hence $v_i^TF_2v_i \geq 0$.
So we have
\begin{align*}
& F(\eps) \bullet X  = \sum_{i=1}^r u_iF(\eps)u_i \\
& = \sum_{i=1}^k \tau_i^2 (f(v_i/\tau_i)+\eps (1+\|v_i/\tau_i\|_2^2)) +
\sum_{i=k+1}^r ( v_i^TF_2v_i+\eps  \|v_i\|_2^2)  \\
& \geq \sum_{i=1}^k \tau_{i}^{2} \eps  +
\sum_{i=k+1}^r  \eps \|v_i\|_2^2.
\end{align*}
Since every $u_i$ is nonzero, we have either $\tau_i >0$
or $v_i \ne 0$. Thus we must have
\[
F(\eps) \bullet X >0,
\]
which contradicts that $F(\eps) \bullet X \leq \eta \leq 0$.
So $F(\eps)$ must belong to $\mc{E}$, and the result follows.

(b) If there exist $\xi, \zeta$ such that $ q(\xi) >0 > q(\zeta)$,
then we are done by applying Proposition~\ref{prop:S-lm-E(q)}.
Replacing $q$ by $-q$ if necessary, we may thus assume that
 $q(x)\geq 0$ for all $x\in \re^n$.
Let us recall the decomposition
$q(x) = q_{0} + 2q_1^Tx + x^TQ_{2}x$ given just before the proof of (a).
Since $E(q)\ne \emptyset$, we may assume that the origin
belongs to $E(q)$, i.e., $q_0=q(0)=0$.
Since we assumed $q(x)\geq 0$ for all $x\in \re^n$,
the origin is a minimizer of $q(x)$,
whence $\nabla q(0) = 0$. Thus it follows that
\[
q_1 = \half \nabla q(0) = 0.
\]
We now proceed to derive a contradiction similar to that used in
(a), but now we define
 $\mc{E}$ as
\[
\mc{E} = \{ S + t Q:\, S^T=S\succeq 0, t \in \re \}.
\]
As in part (a), if $F(\eps) \not\in \mc{E}$,
then there exist a nonzero symmetric matrix
$X$ and a scalar $\eta$ such that
\[
F(\eps) \bullet X \leq \eta, \quad
(S+t Q)  \bullet X \geq \eta, \, \forall \, S^T=S\succeq 0, \, \forall \, t \in \re,
\]
which implies
\[
Q \bullet X = 0, \quad  X \succeq 0, \quad \eta \leq 0.
\]
Again, applying Proposition~\ref{prop:SZ-decomp},
we get nonzero vectors $u_1, \ldots, u_r$ such that
\[
X=u_1u_1^T+\cdots+u_ru_r^T, \quad
u_1^TQu_1 = \cdots = u_r^TQu_r =  \frac{Q\bullet X}{r} = 0.
\]
As before, write  $u_i$ as
\[
u_i = \bbm \tau_i \\ v_i \ebm,
\]
and reorder the  $u_i$ so that
$\tau_{i} \ne 0$ ($1\le i\le k$), and $ \tau_{k+1} = \cdots =\tau_r = 0$.
For $i=1,\ldots,k$, we have
\[
 \tau_i^2 q(v_i/\tau_i)= u_i^TQu_i = 0,
\]
so $v_i/\tau_i \in E(q)$, and hence $f( v_i / \tau_i) \geq 0$.
For every $i=k+1,\ldots,r$, we have that for all $\af \in \re$,
\[
0 =  \af^2 u_i^TQu_i = \af^2 v_i^TQ_2v_i = q(\af v_i).
\]
Thus we get
\[
0\le f (\af v_i) = f_0 + 2 \af f_1^T v_i + \af^2 v_i^TF_2v_i, ~~\forall ~~ \af \in \re,
\]
whence $v_i^TF_2v_i \geq 0$ for $i=k+1,\ldots,r$.
As in part (a), we have
\begin{align*}
F(\eps) \bullet X
& = \sum_{i=1}^k \tau_i^2 (f(v_i/\tau_i)+\eps (1+\|v_i/\tau_i\|_2^2)) +
\sum_{i=k+1}^r ( v_i^TF_2v_i+\eps  \|v_i\|_2^2)  \\
& \geq \sum_{i=1}^k \tau_{i}^{2}\eps  +
\sum_{i=k+1}^r \eps \|v_i\|_2^2 >0,
\end{align*}
which contradicts  $F(\eps) \bullet X \leq 0$.
So $F(\eps)$ must belong to $\mc{E}$, and the result follows.
\end{proof}

\subsection{Quadratic moment problems}

We now apply the preceding results to quadratic moment problems. Recall
from \cite{CF96} that for $n=1,~2$, if $M_{1}(y)\succeq 0$, then
$M_{1}(y)$ has a flat extension, and thus
$y$ has admits a $rank~M_{1}(y)$-atomic representing measure.
We begin by generalizing the latter result to $n\ge 1$.

\begin{thm} \label{thm:deg2mom}
If $y \in \mc{M}_{n,2}$ and $M_1(y) \succeq 0$,
then $y$ has a $rank~M_1(y)$-atomic representing measure.
\end{thm}
\begin{proof}
Without loss of generality, we may normalize $y$ so that $y_{0}=1$.
Write the moment matrix $M_1(y)$ as follows:
\[
M_1(y) = \bbm 1 & z^T \\ z & W \ebm,
\]
where $z\in \re^{n}$. Since $y_0 = 1$,
we can choose a number $\af>0$ small enough such that the matrix
\[
Q = \bbm 1 & 0 \\ 0 & -\af I_n \ebm
\]
satisfies $Q  \bullet M_1(y) \geq 0$.
Then, by Proposition~\ref{prop:SZ-decomp},
there exist nonzero (column)
vectors $u_1,\ldots,u_r\in \re^{n+1}$ ($r=rank~M_1(y)$)
such that
\[
M_1(y) = u_1u_1^T+\cdots+u_ru_r^T, \quad
u_1^TQu_1 = \cdots = u_r^TQu_r = \frac{Q\bullet M_1(y)}{r} \geq 0.
\]
Write the vectors $u_i$ as
\[
u_i = \bbm \tau_i \\ w_i \ebm,  \tau_i \in \re, w_i \in \re^n.
\]
Then $u_i^TQu_i \geq 0$ implies that $\tau_i^2 \geq \af \|w_i\|_2^2$.
So, if $\tau_i =0$, then $w_i=0$.
Note that $\|u\|_i^2 = \tau_i^2 + \|w_i\|_2^2$.
Since all $u_i$ are nonzero, every $\tau_i \ne 0$,
and hence we can write $u_i$ as
\[
u_i = \tau_i \bbm 1 \\ v_i \ebm,  v_i \in \re^n.
\]
Thus, we have
\be  \label{rep:r-atom}
M_1(y) = \tau_1^2 \bbm 1 \\ v_1 \ebm \bbm 1 \\ v_1 \ebm^T + \cdots +
\tau_r^2 \bbm 1 \\ v_r \ebm \bbm 1 \\ v_r \ebm^T.
\ee
The above gives a $r$-atomic representing measure for $y$.
\end{proof}

We pause to give an application of Theorem~\ref{thm:deg2mom}
to the multivariable degree one moment problem.
\begin{cor} A degree one multisequence  $y$ has
a representing measure if and only if $y_{0} > 0$.
\end{cor}
\begin{proof} Note that if $v$ denotes the vector of
moments in $y$, in degree-lexicographic order, then
$v^{T}v$ has the form of a positive moment matrix $M_{1}$,
so the existence of a representing measure follows from
 Theorem~\ref{thm:deg2mom} .
\end{proof}

We next turn to the quadratic $K$-moment problem where
$q$ is a quadratic polynomial and $K= E(q)$ or $K=S(q)$.
For the case when $n = 2$ and
$q(x) = 1 - \|x\|_{2}^{2}$, it is known that the conditions
$M_{1}(y) \succeq 0$ and $L_{y}(q) = 0$ (resp., $L_{y}\ge 0$)
imply the existence of representing measures supported
in $E(q)$ \cite[Theorem~3.1]{CF00}  (resp., $S(q)$
\cite[Theorem 1.8]{CF00}).  This can be generalized
to $n\ge 1$ and $S(q)$ compact.

\begin{thm}  \label{thm:2mom-compact}
Suppose $q(x)$ is quadratic and $S(q)$ is compact and nonempty. \\
(a) $y \in \mc{M}_{n,2}$ has a representing measure
supported in $E(q)$ if and only if
\[
M_1(y) \succeq 0, \quad L_y(q) = 0.
\]
(b) $y \in \mc{M}_{n,2}$ has a representing measure
supported in $S(q)$ if and only if
\[
M_1(y) \succeq 0, \quad L_y(q) \geq 0.
\]
\end{thm}
\begin{proof}
We write $q(x)$ as
\be \label{eq:q(x)dcomp}
q(x) = q_0 + 2q_1^Tx + x^TQ_2x =
\bbm 1 \\ x \ebm^T
\underbrace{\bbm q_0 & q_1^T \\ q_1 & Q_2 \ebm}_{Q} \bbm 1 \\ x \ebm.
\ee
Since $S(q)$ is nonempty, we can assume $0\in S(q)$, i.e., $q_0\geq 0$,
without loss of generality.
From the compactness of $S(q)$, we know $q(x)$ must be strictly concave, that is,
$Q_2$ must be negative definite ($Q_2\prec 0$).
To see this, suppose otherwise, i.e., that $Q_2$ is not negative definite.
Then there exists a nonzero $u \in \re^n$ such that $u^TQ_2u\geq 0$.
We can also further choose $u$ so that $q_1^Tu\geq 0$ (otherwise replace $u$ by $-u$).
Thus, for any $t>0$, we have $q(tu)\geq 0$, which implies
$S(q)$ is unbounded. However, this contradicts the compactness of $S(q)$.
Therefore, $Q_2$ must be negative definite.

(a)
We need only prove the sufficiency direction.
Suppose $y \in \mc{M}_{n,2}$ and let $X=M_1(y)$.
Then we have
\[
X\succeq 0, \quad Q \bullet X = L_y(q) = 0.
\]
By Proposition~\ref{prop:SZ-decomp},
there exist nonzero vectors $u_1,\ldots, u_r \in \re^{n+1}$ such that
\[
X = \sum_{i=1}^r u_iu_i^T, \quad
u_1^TQu_1 = \cdots = u_r^TQu_r = \frac{Q \bullet X}{r} = 0.
\]
Write $u_i = \bbm \tau_i & w_i^T \ebm^T$ for some scalar $\tau_{i}$
and some vector $w_i \in \re^n$.
Then $u_i^TQu_i = 0$ implies that
\be \label{eq:w_iQw_i=0}
q_0 \tau_i^2 + 2 \tau_i q_1^Tw_i + w_i^TQ_2w_i =0.
\ee
If $\tau_i=0$ for some $i$, then $w_i^TQ_2w_i=0$, and hence $w_i=0$
because of negative definiteness of $Q_2$.
Since $u_i$ is nonzero,
it follows that every $\tau_i \ne 0$, and we can write
$u_i = \tau_i [1 \,\,  v_i^T]^T$. \reff{eq:q(x)dcomp} and \reff{eq:w_iQw_i=0}
now imply that $q(v_i) = 0$, so $v_i \in E(q)$.
Therefore, we have
\[
M_1(y) = \tau_{1}^{2} \bbm 1 \\ v_1 \ebm \bbm 1 \\ v_1 \ebm^T + \cdots +
\tau_{r}^{2} \bbm 1 \\ v_r \ebm \bbm 1 \\ v_r \ebm^T,
\]
and it follows that $\mu \equiv \sum_{i=1}^{r} \tau_{i}^{2} \delta_{v_{i}}$
is a representing measure for $y$ supported in $E(q)$.

(b)The proof is very similar to part (a).
Suppose $y \in \mc{M}_{n,2}$ and
let $X=M_1(y)$. Then
\[
X\succeq 0, \quad Q \bullet X = L_y(q) \geq 0.
\]
By Proposition~\ref{prop:SZ-decomp},
there exist nonzero vectors $u_1,\ldots, u_r \in \re^{n+1}$ such that
\[
X = \sum_{i=1}^r u_iu_i^T, \quad
u_1^TQu_1 = \cdots = u_r^TQu_r = \frac{Q \bullet X}{r}\geq 0.
\]
Write $u_i = \bbm \tau_i & w_i^T \ebm^T$ for some $w_i \in \re^n$.
Then $u_i^TQu_i \geq 0$ implies that
\be  \label{eq:qw_i>=0}
q_0 \tau_i^2 + 2 \tau_i q_1^Tw_i + w_i^TQ_2w_i \geq 0.
\ee
If $\tau_i=0$ for some $i$, then $w_i^TQ_2w_i\geq 0$ and hence $w_i=0$
because of negative definiteness of $Q_2$.
But this is also impossible, since otherwise $u_i=[\tau_i\,\,\, w_i^T ]^T$ is a zero vector.
Thus, every $\tau_i\ne 0$. So we can further write
$u_i = \tau_{i} \bbm 1 & v_i^T \ebm^T$.
Then \reff{eq:q(x)dcomp} and \reff{eq:qw_i>=0} imply
that $q(v_i) \ge 0$, and so $v_i \in S(q)$.
Hence we get
\[
M_1(y) = \tau_{1}^{2} \bbm 1 \\ v_1 \ebm \bbm 1 \\ v_1 \ebm^T + \cdots +
\tau_{r}^{2} \bbm 1 \\ v_r \ebm \bbm 1 \\ v_r \ebm^T,
\]
and it follows as above that $y$ has a
representing measure supported in $S(q)$.
\end{proof}

When $E(q)$ or $S(q)$ is not compact,
the conclusions of Theorem~\ref{thm:2mom-compact} might fail.
However, we can get a sightly weakened version.

\begin{thm} \label{2mom-clos}
Let $y \in \mc{M}_{n,2}$ and let $q(x)$ be a quadratic polynomial.\\
\indent (i) Suppose $E(q)\ne \emptyset$. Then $M_1(y) \succeq 0$ and $L_y(q) = 0$
if and only if $y \in \overline{\mc{R}_{n,2}(E(q))}$. \\
\indent (ii) Suppose $S(q)\ne \emptyset$. Then $M_1(y) \succeq 0$ and $L_y(q) \geq 0$
if and only if $y \in \overline{\mc{R}_{n,2}(S(q))}$.
\end{thm}

\begin{proof}
(i) The sufficiency direction is obvious, so we only need prove necessity.
Suppose to the contrary that  $M_1(y) \succeq 0$ and $L_y(q) = 0$, but
$y \not\in \overline{\mc{R}_{n,2}(E(q))}$.
Since $\overline{\mc{R}_{n,2}(E(q))}$ is a closed convex cone,
Minkowski's separation theorem implies that
there exists a nonzero polynomial $p\in \mathcal{P}_{2}$
such that
\[
L_{y}(p)\equiv \hat{p}^T y < 0, \quad \text{and} \quad
\hat{p}^T w \geq 0, \, \forall \, w\in \overline{\mc{R}_{n,2}(E(q))}.
\]
For $1\le i\le n$, let $y_{2e_{i}}$ denote the element of $y$
corresponding to the monomial $x_{i}^{2}$.
Choose $\eps>0$ small enough so that
\begin{equation} \label{eq:hatp+eps<0}
\hat{p}^T y + \eps ( 1 + \sum_{i=1}^n y_{2e_i})L_{y}(1) < 0,
\end{equation}
and define the nonzero polynomial
\begin{equation} \label{eq:pt=phat+eps}
\tilde{p}(x) = \hat{p}^T [x]_2 + \eps (1+\|x\|_2^2).
\end{equation}
Since,
  for each $x\in E(q)$,
 the monomial vector $[x]_{2}$ belongs to $ \mc{R}_{n,2}(E(q))$
(with $E(q)$-representing measure $\delta_{x}$),
the polynomial $\hat{p}^{T}[x]_{2}$ is nonnegative on $E(q)$.
By Proposition~\ref{prop:F+eps}-(b), there exists
$t\in \re$  such that
\[
\hat{p}^{T}[x]_{2}  + \eps (1+\|x\|_2^2) -t q(x)\ge 0 ~~ \forall x\in \re^n.
\]
It follows from \reff{p=[]1^TP[]1} that there exists
a matrix $P$, with $P=P^T \succeq 0$, such that
\[
\hat{p}^{T}[x]_{2}  + \eps (1+\|x\|_2^2) -t q(x) =  [x]_1^TP[x]_1
 ~~ \forall x\in \re^n,
\]
whence
\[
\tilde{p}(x) = [x]_1^TP[x]_1 + t q(x).
\]
Since
$M_{1}(y)\succeq 0$ and $L_{y}(q) = 0$,
applying $L_{y}$ on both sides of the above
(see equation \reff{Lyp=PdotM1(y)}) implies that
\[
L_y(\tilde{p}) =  P \bullet M_1(y)  + t L_y(q) = P \bullet M_1(y)  \geq 0.
\]
However, from \reff{eq:hatp+eps<0}-\reff{eq:pt=phat+eps} we have
\[
L_y(\tilde{p}) = \hat{p}^{T}y + \eps (1 + \sum_{i=1}^n y_{2e_i} )L_{y}(1) < 0,
\]
which is a contradiction.
So we must have
$y \in \overline{\mc{R}_{n,2}(E(q))}$.

(ii) Sufficiency is again obvious, so we focus on necessity.
The proof is very similar to the argument of (i),
but we replace $E(q)$ by $S(q)$.
Thus, the polynomial $\hat{p}^{T}[x]_{2}$ is now nonnegative on $S(q)$.
Using Proposition 4.4-(a), it follows as above that
there exists $t\ge 0$  and a matrix $P$ with $P = P^{T}\succeq 0$,
such that   $\tilde{p}(x) = [x]_1^TP[x]_1 + t q(x)$. Since $t\ge 0$
and $L_{y}(q)\ge 0$, it follows as before that $L_{y}(\tilde{p})\ge 0$,
which leads to the same contradiction as in (i).
\end{proof}

Theorem~\ref{2mom-clos} implies that if $q$ is a quadratic polynomial
and if $M_1(y) \succeq 0$ and $L_y(q)=0$ (resp. $L_y(q)\geq 0$),
then $y$ is in the closure of the quadratic moment sequences  which admit
representing measures supported in $E(q)$ (resp. $S(q)$).
But this does not necessarily imply that
$y$ admits a representing measure supported in $E(q)$ or $S(q)$,
as the following example shows.

\begin{exm}
Let $n=2$ and let $y \in \mc{M}_{2,2}$ be the quadratic
moment sequence such that
\[
M_1(y) = \bbm 1 & 1 & 1 \\ 1 & 1 & 1\\ 1 & 1 & 2 \ebm.
\]
Let $1$, $X_{1}$, $X_{2}$ denote the columns of $M_{1}(y)$.
Obviously, $M_1(y)$ is positive semidefinite with $rank~M_{1}(y)=2$,
so $y$ admits $2$-atomic representing measures by Theorem 4.5.
Since $1=X_{1}$,
 Proposition~3.1 of \cite{CF96} implies that any representing measure $\mu$
must be supported in the variety $\{(x_{1},x_{2})\in \re^2: x_{1} =1 \}$.

 Let $q(x) = x_2 -x_1^2$.
Then $S(q)$ is convex but noncompact,
and $E(q)$ is nonconvex and noncompact.
Note that $L_y(q) = y_{01} - y_{20} = 0$, so of course $L_y(q) \geq 0$.
But $y$ does not have a representing measure $\mu$ supported in
either $E(q)$ or $S(q)$.
Indeed, suppose a
representing measure $\mu$ with
 $supp~\mu \subseteq S(q)$ exists.
For any $x=(x_1,x_2) \in supp~\mu \subseteq S(q)$,
we must have $x_1 = 1$ and $x_2 \geq 1$.
Then the relation
\[
\int_{\re^2} x_2 d\mu(x) = y_{01} = 1,
\]
together with $y_{00}=1$,
implies that $x_2=1$ on the support of $\mu$.
So $\mu$ is supported at the single point $(1,1)$,
which is obviously false.
Therefore, $y$ does not have a representing measure $\mu$
supported in $S(q)$ or $ E(q)$.

In keeping with Theorem 4.7, we next show that
an arbitrarily small perturbation can be applied to make
the perturbed $y$ have a representing measure supported in
$E(q) (\subset S(q))$.
For $1> \eps >0$, let
the moment sequence $\bar{y}(\eps)$ be defined by
\[
M_1(\bar{y}(\eps)) =
(1-\eps) \bbm 1 \\ 1 \\ 1 \ebm \bbm 1 \\ 1 \\ 1 \ebm^T +
\eps \bbm 1 \\ \eps^{-1/4} \\  \eps^{-1/2} \ebm
\bbm 1 \\ \eps^{-1/4} \\  \eps^{-1/2} \ebm^T
\]
\[ = \bbm 1 & 1-\eps + \eps^{3/4} & 1+\eps^{1/2}-\eps \\
          1-\eps + \eps^{3/4}  & 1+ \eps^{1/2}-\eps & 1+\eps^{1/4}-\eps \\
          1 + \eps^{1/2}-\eps  & 1+\eps^{1/4}-\eps & 2-\eps \ebm .
\]
We  see that $\bar{y}(\eps) \to y$ as $\eps \to 0$,
and $\bar{y}(\eps)$ has the $2$-atomic $E(q)$-representing measure
$$(1-\eps) \delta_{(1,1)} + \eps \delta_{(\eps^{-\frac{1}{4}},\eps^{-\half})}.$$
\qed
\end{exm}

Despite the preceding example,
 if, in Theorem~\ref{2mom-clos}, the quadratic moment sequence $y$
is such that $M_1(y) \succ 0$ and
$L_y(q)=0$ (resp. $L_y(q) >0$),
then $y$ does have a representing measure supported in $E(q)$ (resp. $S(q)$).
The following result thus provides some affirmative evidence
for Question 1.2.


\begin{thm} \label{2mom-int=>rep}
Let $y \in \mc{M}_{n,2}$ and let $q(x)$ be a quadratic polynomial.\\
(i) If $E(q)\ne \emptyset$, $M_1(y) \succ 0$ and $L_y(q) = 0$,
then $y \in \mc{R}_{n,2}(E(q))$. \\
(ii) If $S(q)\ne \emptyset$, $M_1(y) \succ 0$ and $L_y(q) > 0$,
then $y \in \mc{R}_{n,2}(S(q))$.
\end{thm}
\begin{proof}
(i) Define the affine subspace $\mc{N}(q)$ and set $\mc{F}_E$ as follows:
\[
\mc{N}(q) = \{ y \in \mc{M}_{n,2}:   L_y(q) = 0 \}, \quad
\mc{F}_E=\{ y \in \mc{N}(q):  M_1(y) \succeq 0\}.
\]
Note that $\mc{R}_{n,2}(E(q))$ and $\mc{F}_E$ are both convex sets
contained in the space $\mc{N}(q)$.
Theorem~\ref{2mom-clos} says that $\mc{F}_E = \overline{\mc{R}_{n,2}(E(q))}$.
If $M_1(y) \succ 0$,
then $y$ lies in the interior of $\mc{F}_E$.
By Lemma~\ref{lem:clos-int}, we know $y \in \mc{R}_{n,2}(E(q))$.

(ii) Let $\mc{F}_S$ be the following convex set
\[
\mc{F}_S=\{ y \in \mc{M}_{n,2}:  M_1(y) \succeq 0, L_y(q) \geq 0 \}.
\]
Theorem~\ref{2mom-clos} says that $\mc{F}_S = \overline{\mc{R}_{n,2}(S(q))}$.
If $M_1(y) \succ 0$ and $L_y(q)>0$,
then $y$ lies in the interior of $\mc{F}_S$.
Hence Lemma~\ref{lem:clos-int} implies $y \in \mc{R}_{n,2}(S(q))$.
\end{proof}

Using Theorem~\ref{2mom-int=>rep}, we can now show that Question 1.2
has an affirmative answer when $d=1$ and $K=E(q)$ or $K=S(q)$
for a quadratic polynomial $q(x)$.

\begin{corollary}
Let $y \in \mc{M}_{n,2}$ and let $q(x)$ be a quadratic polynomial.  \\
(i) Suppose $E(q) \ne \emptyset$.  If $M_{1}(y)\succ 0$ and $L_{y}$ is
$E(q)$-positive, then
$y$ has an $E(q)$-representing measure. \\
(ii)  Suppose $S(q) \ne \emptyset$. If $M_{1}(y)\succ 0$ and $L_{y}$ is
$S(q)$-positive, then
$y$ has an $S(q)$-representing measure.
\end{corollary}
\begin{proof}
(i) From Theorem~\ref{2mom-int=>rep} (i),
it suffices to show that $L_{y}(q) = 0$.
Since $L_{y}$ is $E(q)$-positive, we have $L_{y}(q) \ge 0$ and
$L_{y}(-q) \ge 0$, so $L_{y}(q) = 0$.

(ii)
Suppose first that $E(q) \ne \emptyset$.
Since $L_{y}$ is $S(q)$-positive, $L_{y}(q) \ge 0$. If
$L_{y}(q) = 0$, Theorem~\ref{2mom-int=>rep} (i)
implies that y has a representing measure supported in
$E(q) \subseteq S(q)$.
If $L_{y}(q) > 0$, then Theorem~\ref{2mom-int=>rep} (ii) shows that $y$ has a
representing measure supported in $S(q)$.
Suppose next that $E(q) = \emptyset$. Since $S(q) \ne \emptyset$,
then $S(q) = \re^n$, so in this case the result follows from
Theorem 4.5.
\end{proof}

%
%
%

\end{document}